\documentclass{amsart}

\usepackage{amssymb}
\usepackage{color}
\usepackage{graphicx}
\usepackage{cite}

\newcommand{\rmd}{\mathrm{d}}

\DeclareMathOperator{\Imn}{Im}

\newcommand{\Jac}{\mathrm{Jac}}
\newcommand{\Kum}{\mathrm{Kum}}
\newcommand{\Sym}{\mathrm{Sym}}

\DeclareMathOperator{\Complex}{\mathbb{C}}

\DeclareMathOperator{\Integer}{\mathbb{Z}}
\DeclareMathOperator{\Natural}{\mathbb{N}}
\DeclareMathOperator{\wgt}{wgt}

\newcommand{\Pic}{\mathrm{Pic}}

\DeclareMathOperator{\add}{\mathsf{add}}
\DeclareMathOperator{\inv}{\mathsf{inv}}

\DeclareMathOperator{\rank}{rank}
\DeclareMathOperator{\res}{res}

\newcommand{\J}{\mathcal{J}}

\newcommand{\mFr}{\mathfrak{m}}

\newcommand{\F}{\mathrm{F}}

\newtheorem{Theorem}{Theorem}

\newtheorem{Corollary}{Corollary}
\newtheorem{Property}{Property}

\newtheorem{Conjecture}{Conjecture}

\theoremstyle{definition}
\newtheorem{Definition}{Definition}
\newtheorem{Remark}{Remark}

\allowdisplaybreaks[1]

\title[]{Abelian function fields on Jacobian varieties}
\author{Julia Bernatska}
\address{}
\email{jbernatska@gmail.com}
\date{\today}

\allowdisplaybreaks[4]

\begin{document}
 
\maketitle 
\begin{abstract}
In this paper
the fields of multiply periodic, or Kleinian $\wp$-functions are exposed.
Such a field arises on the Jacobian variety of an algebraic curve, and provides natural algebraic models of the 
Jacobian and Kummer varieties, possesses the addition law, and accommodates dynamical equations with solutions. 
All this will be explained in detail for plane algebraic curves in their canonical forms. 
Example of hyperelliptic and non-hyperelliptic curves are presented.
\end{abstract}

\section{Introduction}
\footnote{This expository paper was written for 
the special issue of Axioms entitled "Recent Advances in Function Spaces and Their Applications"}  In this paper we consider differential fields of multiply-periodic meromorphic  functions 
defined on  Jacobian varieties. It is essential that every Jacobian variety under consideration
arises from an algebraic curve.

We work with plane algebraic curves,
for which the theory of modular-invariant, entire function, known as multi-variable (or multi-dimensional) $\sigma$-function,
and multiply-periodic meromorphic functions, known as Kleinian $\wp$-functions, 
was productively developed in the last thirty years, after it was abandoned in the beginning of the 20-th century.
We frequently recall the results published in the classical monographs of H. F. Baker
\cite{bakerAF, bakerMPF}. The former was written as a guide `to analytical developments of Pure Mathematics
during the last seventy years' of the 19-th century,
and the latter as `an elementary and self-contained introduction to many of the leading ideas of the theory
of multiply-periodic functions', illustrated by an example of hyperelliptic functions in genus two
which `reduces the theory in a very practical way'.

The results obtained during the last thirty years are widely presented in reviews \cite{bel2012, bel2020}.
A review on a more general algebraic approach can be found in \cite{KMP2022}. 
The present review pursues, first of all, the practical purpose: to expose how to 
work with abelian functions on the Jacobian variety of a given plane algebraic curve.

We represent abelian functions through $\wp$-functions, 
which generalize the Weierstrass $\wp$-function to higher genera.
$\wp$-Functions associated with a curve of genus $g$ are periodic on
the period lattice of rank $2g$ which naturally arises on the curve. 
Similar to the Weierstrass $\wp$-function, these functions
are generated from $\sigma$-function, which is entire and modular-invariant with respect to the  period lattice.
The concept of $\sigma$-functions in higher genera was suggested by Klein \cite{klein1886,klein1888,klein1890}.
Therefore, $\wp$-functions in $g$ variables are called Kleinian in \cite{belHKF1996},
when the results obtained in  the 19-th century were recalled after almost a century-long lull,
and extended.

Due to the works of K.~Weierstrass~\cite{weier1903}, O.~Bolza~\cite{bolza1895,bolza1899},
H.~F.~Baker ~\cite{bakerHE1898,bakerDE1903}, further progress in developing the theory of 
multi-variable $\sigma$-functions and multiply-periodic $\wp$-functions associated with hyperelliptic curves
was made. In the end of the 20-th -- the beginning of the 21-st centuries  
V. Enolski, V. Leykin, and V. Buchstaber pushed progress to a new level.  
The theory of constructing heat equations which define 
multi-variable $\sigma$-function as an analytic series was developed \cite{bel1999,BL2002,BL2004,BL2008},
as well as other techniques for working with $\wp$-functions \cite{belHKF1996,BEL1997,EEL2000,bel2000,BL2005,BL2005ru}.
This gave birth to an extensive research on identities for $\wp$-functions associated with
hyperelliptic and non-hyperelliptic curves \cite{AEE2003,AEE2004,ath2008,ath2011,ath2012,BG2006,BEGO2008,EEP2003,
EEMOP2007,EEMOP2008,EEG2010,EEO2011,EGOY2017,EE2009,EG2009,AE2012,E2010,E2011}. 
At the same time, other aspects of the theory of multi-variable $\sigma$-functions were developed in
\cite{NakS2010, Nak2011, NakST2010, NakTT2016, KMP2022}, and
an extension to space curves in \cite{MK2013, KMP2013,KMP2019}.

\bigskip
\textbf{List of notations}

\begin{tabular}{ll}
$n$,  $s$ & co-prime natural numbers \\
 $\Natural_0 = \Natural \cup \{0\}$ & non-negative integers  \\
$\equiv$ & means `defined by' \\
$\mathcal{C}$ & an algebraic curve, assumed to be in the canonical form \\
$\mathcal{C}^n$ & the $n$-th symmetric product of $\mathcal{C}$ \\
$g$ & the genus of $\mathcal{C}$ \\
$\lambda$ & a vector of parameters ($=$ coefficients of the equation) of $\mathcal{C}$\\
 $\mathfrak{W} = \{ \mathfrak{w}_i \mid i=1, \dots, g \}$ & a Weierstrass gap sequence \\
 $\mathfrak{M} = \{ \mathfrak{m}_{w} \mid w \in \Natural_0  \backslash \mathfrak{W} \}$ 
 & an ordered list of monomials on $\mathcal{C}$ \\
$\omega$, $\omega'$ & first kind not normalized $\mathfrak{a}$-, and $\mathfrak{b}$-period matrices \\
$\eta$, $\eta'$ & second kind not normalized $\mathfrak{a}$-, and $\mathfrak{b}$-period matrices \\
$\Jac (\mathcal{C}) = \Complex^g / \{\omega,\omega'\}$ & the Jacobian variety of $\mathcal{C}$, w.r.t. not normalized periods  \\
$\Kum (\mathcal{C}) = \Jac (\mathcal{C}) / \pm$ & the Kummer variety of $\mathcal{C}$ \\
$\rmd u_{\mathfrak{w}_i} = \upsilon_{\mathfrak{w}_i} \rmd x / \partial_y f$ & 
first  kind (or holomorphic) differentials on $\mathcal{C}$ \\
$\rmd r_{\mathfrak{w}_i} = \rho_{\mathfrak{w}_i} \rmd x / \partial_y f$ & 
second kind  differentials on $\mathcal{C}$ \\
$\mathcal{A}(P)$, $\mathcal{A}(D)$ & the Abel image (or first kind integral) of a point $P$ \\
& and a divisor $D$ \\
$\mathcal{B}(P)$, $\mathcal{B}(D)$ &  the second kind integral at a point $P$, and a divisor $D$ \\
$\Sigma$ & the theta-divisor defined by $\{u \in \Jac(\mathcal{C}) \mid \sigma(u)=0\}$ \\
$\mathfrak{A}(\mathcal{C})$ & differential field of $\wp$-functions on  $\Jac(\mathcal{C}) \backslash \Sigma$ \\
$\mathcal{R}_w$ & a polynomial function of weight $w$ from $\Complex[x,y] / f(x,y;\lambda)$ \\
$\mathfrak{P}(\mathcal{C})$ & the vector space of polynomial functions on $\mathcal{C}$ \\
$\upsilon_{\mathfrak{w}_1}$,  \ldots, $\upsilon_{\mathfrak{w}_g}$ & basis monomials in $\mathfrak{P}(\mathcal{C})$
\end{tabular}

\section{Preliminaries}
\subsection{Canonical form of plane algebraic curves}
As canonical forms of plane curves we use so called $(n,s)$-curves, $\gcd(n,s)=1$, introduced in \cite{bel1999}
as follows
\begin{align}\label{nsCurve}
\mathcal{C} =\{(x,y)\in \Complex^2 \mid f(x,y;\lambda) \equiv 
-y^n + x^s + \sum_{j=0}^{n-2} \sum_{i=0}^{s-2} & \lambda_{ns-in- js} y^j x^i =0, \\
& \lambda_{k\leqslant 0}=0, \ \  \lambda_{k}\in \Complex\}. \notag
\end{align}
The curve equation in \eqref{nsCurve} arises as a universal unfolding of the Pham singularity $y^n + x^s = 0$,
and contains the minimal number of parameters $\lambda_k$, which is $\mathfrak{N} = (n-1)(s-1)-M$, where $M$
is the number of $\lambda_{k\leqslant 0}$, called modality. By $\lambda$ we denote the list of 
all parameters $\lambda_k$ of a curve~$\mathcal{C}$.

All curves \eqref{nsCurve} with $\lambda \in \Lambda \equiv \Complex^{\mathfrak{N}}$
form a fiber bundle $\mathcal{E}(\Lambda,\pi,\mathcal{C})$ with 
the projection $\pi(\mathcal{C})=\lambda$.
Genera of curves in $\mathcal{E}(\Lambda,\pi,\mathcal{C})$ do not exceed, see \cite{bel1999},
\begin{equation}\label{gDef}
g=\tfrac{1}{2}(n-1)(s-1), 
\end{equation}
Thus, $\mathfrak{N} = 2g-M$.
The space of parameters $\Lambda$,
which serves as the base, is 
naturally stratified into $g+1$ strata: $\Lambda= \cup_{k=0}^g \Lambda_k$,
such that curves  with  $\lambda \in \Lambda_k$ have the actual genus $k$.
In \cite{BerLey2019} such a stratification in the case of a genus 2 curve is described in detail.
The whole theory works for $\lambda$ from all strata. 

An $(n,s)$-curve possesses the Weierstrass gap sequence generated by $n$ and $s$, namely
$$ \mathfrak{W} = \Natural_0 \backslash \{a n+b s \mid a, b \in \Natural_0 \}.$$
The length of this sequence equals \eqref{gDef}. In the presence of double points, 
the genus of the curve decreases, that truncates the gap sequence.

\begin{Remark}
In fact, $(n,s)$-curves are Weierstrass canonical forms, see \cite[Chap.\,V]{bakerAF}; they
generalize the Weierstrass canonical form
 $-y^2 + 4 x^3 - g_2 x - g_3 = 0$ of elliptic curves. 
Every plane algebraic curve can be bi-rationally transformed into the one with a branch point at infinity
where all $n$ sheets wind, \cite[p.\,92]{bakerAF}.
And this is an $(n,s)$-curve, possibly with double\footnote{A double point
 on a Riemann surface refers to a point where two distinct sheets of the Riemann surface come together.
If more than two sheets come together at such a point, we say, 
that two or more double points are located at this place.} points.
\end{Remark}
Infinity on \eqref{nsCurve} is a Weierstrass point, and
 serves as the basepoint of the Abel map.

In what follows, we work with $\mathcal{E}(\Lambda_g,\pi,\mathcal{C})$,
that is we  assume that a curve  $\mathcal{C}$ has genus equal to $g$ computed by \eqref{gDef}.

\begin{Remark}
In the presence of one or more double points, the actual genus $\tilde{g}$ is less than $g$, say  $\tilde{g}=g-k$.
In this case, we work with $\lambda$ from a subspace of $\Lambda_{g-k}$
defined by the relations between $\lambda$ which introduce double points.
We call such curves degenerate. The case of a degenerate curve
can be derived from the case of the maximal genus, 
as shown in \cite{BerLey2019}.
\end{Remark}

One can add extra terms in \eqref{nsCurve}, namely $\lambda_{s - n i} y^{n-1} x^i$, $i=1$,\ldots, $[s/n]$, 
and $\lambda_3 x^{s-1}$, which do not affect the genus. For example, a $(3,4)$-curve with extra terms 
is worked out in \cite{EEMOP2007}. Though it seems that introducing extra terms allows to cover
a wider variety of curves, only parameters $\lambda$ included into \eqref{nsCurve} 
serve as independent arguments of $\sigma$-function. Extra terms can be 
injected by a proper bi-rational transformation of \eqref{nsCurve}.

Let  $\mathcal{P}$ be a polynomial in $x$ of degree $s$.  A curve defined by the equation
\begin{equation}
 f(x,y;\lambda) \equiv -y^n + \mathcal{P}(x) = 0,
\end{equation}
 is called a \emph{cyclic} $(n,s)$-curve, 
or \emph{superelliptic}, if $s\geqslant 3 $.

\subsection{Sato weight}
The \emph{Sato weight} plays an important role in the theory of entire and Abelian functions associated with
 $(n,s)$-curves.  The Sato weight shows the negative exponent of the leading term in expansion near
infinity. Let  $\xi$ denote a local parameter in the vicinity of infinity, then $\mathcal{C}$ admits
the following parametrization
\begin{equation}\label{param}
x=\xi^{-n},\qquad y = \xi^{-s} \big(1+O(\lambda)\big).
\end{equation}
Thus, the Sato weights of $x$ and $y$ are $\wgt x = n$, and $\wgt y = s$.  Then $\wgt f(x,y;\lambda) = n s$,
and parameters $\lambda_{k}$ of the curve are assigned with weights:  $\wgt \lambda_k = k$.
Note, that the curve equation in \eqref{nsCurve} contains only parameters with positive Sato weights. 

The Sato weight introduces an order in the space of monomials $\mathfrak{m}_{in+js}(x,y) = x^i y^j$, $j<n$;
the weight of a monomial is indicated in the subscript. We denote by $\mathfrak{M}$ an ordered list of monomials.
Evidently, the weights equal to elements of the Weierstrass gap sequence
$\mathfrak{W} = \{\mathfrak{w}_{i} \mid i=1,\,\dots,\, g\}$ are absent in the list $\mathfrak{M}$.

\subsection{Cohomology basis}
Holomorphic differentials, or differentials of the first kind, in the not normalized form
$\rmd u =$ $(\rmd u_{\mathfrak{w}_1}$, $\rmd u_{\mathfrak{w}_2}$, $\dots$, $\rmd u_{\mathfrak{w}_g} )^t$
are generated by the first $g$ monomials in the  list $\mathfrak{M}$ ordered by the  Sato weight descendingly.
Actually,
\begin{equation}
\rmd u_{\mathfrak{w}_i} = \frac{\mathfrak{m}_{2g-1-\mathfrak{w}_i}(x,y) \, \rmd x}{\partial_y f(x,y;\lambda)},
\quad i=1,\, 2,\, \dots,\, g.
\end{equation}
Weights of $\rmd u$ coincide with the negative Weierstrass gap sequence, namely
 $\wgt u_{\mathfrak{w}_i} = - \mathfrak{w}_i$, which is clearly seen from expansions near infinity.

In addition to first kind differentials,  $g$ differentials of the second kind, or meromorphic differentials
with no simple poles, are required. These second kind differentials are chosen so that they
form an associated system with the first kind differentials, as explained in  
 \cite[Art.\,138]{bakerAF}. The second kind differentials 
 $\rmd r = (\rmd r_{\mathfrak{w}_1} $, $\rmd r_{\mathfrak{w}_2}$, $\dots$, $\rmd r_{\mathfrak{w}_g} )^t$
 have poles at  infinity, of orders equal to elements of the Weierstrass gap sequence, and so 
  $\wgt r_{\mathfrak{w}_i} = \mathfrak{w}_i$.
In the vicinity of infinity, $\xi(\infty)=0$, the following relation holds 
\begin{equation}\label{urRel}
\res_{\xi=0} \Big(\int_0^\xi \rmd u(\tilde{\xi}) \Big) \rmd r(\xi)^t = 1_g,
\end{equation}
where $1_g$ denotes the identity matrix of size $g$. This condition
completely determines the principle part of $\rmd r(\xi)$, which is 
enough for obtaining a solution of the Jacobi inversion problem.
At the same time, derivation of identities for $\wp$-functions requires to specify
$\rmd r$ completely. 

A system of associated first and second kind differentials arise as a part of the process of constructing the 
fundamental bi-differential of the second kind, see \cite[\S\,3.2]{EEL2000}.
In the hyperellitic case, these differentials are given explicitly by \cite[Eq.\,(1.3)]{belHKF1996}.
On an arbitrary $(n,s)$-curve the problem of computing 
fundamental bi-differential is solved in \cite{suz2017}.

In what follows we also use the notation
\begin{gather}\label{DiffNot}
\rmd u_{\mathfrak{w}_i} = \frac{\upsilon_{\mathfrak{w}_i}(x,y)\, \rmd x}{\partial_y f(x,y;\lambda)},\qquad\qquad
\rmd r_{\mathfrak{w}_i} = \frac{\rho_{\mathfrak{w}_i}(x,y)\, \rmd x}{\partial_y f(x,y;\lambda)},
\end{gather}
and call the first $g$ monomials $\upsilon_{\mathfrak{w}_i} (x,y) = \mathfrak{m}_{2g-1-\mathfrak{w}_i}(x,y)$
basis monomials, since they form a basis in the linear space 
$\mathfrak{P}(\mathcal{C})$ of polynomial functions on $\mathcal{C}$.

Let $\{\mathfrak{a}_i,\,\mathfrak{b}_i\}_{i=1}^g$ be  canonical homology cycles  on $\mathcal{C}$.
First kind integrals along these cycles give first kind period matrices (not normalized):
\begin{gather}\label{omegaM}
 \omega = (\omega_{ij})= \bigg( \int_{\mathfrak{a}_j} \rmd u_i\bigg),\qquad\quad
 \omega' = (\omega'_{ij}) = \bigg(\int_{\mathfrak{b}_j} \rmd u_i \bigg).
\end{gather}
Columns of $\omega$, $\omega'$ generate the lattice $\{\omega, \omega'\}$ of periods.
Then $\Jac(\mathcal{C})=\Complex^g/\{\omega, \omega'\}$ is the 
\emph{Jacobian variety}  of  $\mathcal{C}$.
Similarly, second kind period matrices  are defined
\begin{gather}\label{etaM}
 \eta = (\eta_{ij})= \bigg( \int_{\mathfrak{a}_j} \rmd r_i\bigg),\qquad\quad
 \eta' = (\eta'_{ij}) = \bigg(\int_{\mathfrak{b}_j} \rmd r_i \bigg).
\end{gather}

Since $\rmd u$ and $\rmd r$ form an associated system,
period matrices $\omega$, $\omega'$,
$\eta$, $\eta'$  satisfy the Legendre relation, see \cite[Art.\,140]{bakerAF},
\begin{gather}\label{LegRel}
\Omega^t \mathrm{J}\, \Omega = 2\pi \imath \mathrm{J},\\
\Omega = \begin{pmatrix} \omega & \omega' \\
\eta & \eta' \end{pmatrix},\qquad 
\mathrm{J} = \begin{pmatrix} 0 & - 1_g \\ 1_g & 0 \end{pmatrix}. \notag
\end{gather}
The relation \eqref{LegRel} means that $\Omega \in \imath \mathrm{Sp}(2g,\Complex)$,
and $\Omega$ transforms under the action of the symplectic group of size $2g$.

At the same time,  symplectic transformations act on the vector  $\rmd R = \begin{pmatrix} \rmd u\\ \rmd r \end{pmatrix}$
composed of the associated $\rmd u$ and $\rmd r$. This fact singles out these particular $2g$ differentials. 
The vector $\rmd R$ serves as a basis in the space $\mathcal{H}^1(\mathcal{C}^{\circ})$  of holomorphic $1$-forms 
on the curve $\mathcal{C}$ with the puncture at infinity. 
The vector $\rmd R$ is obtained, together with the symplectic gauge 
of the Gauss---Manin connection\footnote{In fact, the Gauss---Manin connection is defined in the bundle
$\mathcal{E}(\Lambda,\varpi,\mathcal{H}^1(\mathcal{C}^{\circ}))$
associated with the bundle $\mathcal{E}(\Lambda,\pi,\mathcal{C})$
of curves.}
in $\mathcal{H}^1(\mathcal{C}^{\circ})$,
in the process of constructing a multi-variable $\sigma$-function, see \cite[\S\,3.1]{BL2008}.
This provides an alternative way 
of obtaining the required second kind differentials.

\subsection{Examples}
{
\theoremstyle{definition}
\newtheorem*{22g1Curve}{$(2,2g{+}1)$-Curves}
\newtheorem*{33m1Curve}{$(3,3\mFr{+}1)$-Curves}
\newtheorem*{33m2Curve}{$(3,3\mFr{+}2)$-Curves}
\newtheorem*{27Curve}{$(2,7)$-Curve}
\newtheorem*{34Curve}{$(3,4)$-Curve}
\newtheorem*{37Curve}{$(3,7)$-Curve}
\newtheorem*{35Curve}{$(3,5)$-Curve}
\newtheorem*{38Curve}{$(3,8)$-Curve}
}
\begin{22g1Curve}
The canonical form of genus $g$ hyperelliptic curves is defined by
\begin{equation}\label{22g1C}
f(x,y;\lambda) \equiv -y^2 + x^{2g+2} 
+ \sum_{i=1}^{2g}  \lambda_{2i+1} x^{2g-i}.
\end{equation}
The Weierstrass gap sequence is $\mathfrak{W} = \{ 2i-1 \mid i=1,\,\dots,\, g\}$.
Associated  first and second kind not normalized differentials are chosen in the form, see \cite[p.\,195 Ex.\,i]{bakerAF},
\begin{subequations}
\begin{align}
&\rmd u_{2i-1} = \frac{x^{g-i} \rmd x}{\partial_y f(x,y;\lambda)},\quad i=1,\,\dots,\,g,\\
&\rmd r_{2i-1} = \frac{\rmd x}{\partial_y f(x,y;\lambda)} \sum_{k=1}^{2i-1} k \lambda_{4i-2k-2} x^{g-i+k},
\quad i=1,\,\dots,\,g. 
\end{align}
\end{subequations}
\end{22g1Curve}
\begin{27Curve}
A $(2,7)$-curve is defined by 
\begin{equation} \label{27C}
f(x,y;\lambda) \equiv -y^2 + x^7 + \lambda_4 x^5 + \lambda_6 x^4
+ \lambda_8 x^3 + \lambda_{10} x^2 + \lambda_{12} x + \lambda_{14}, 
\end{equation}
with the gap sequence $\mathfrak{W}=\{1,3,5\}$, and the ordered list of monomials
\begin{equation}
\mathfrak{M} = \{1,\, x,\, x^2,\, x^3,\, y,\, x^4,\, y x,\, x^5,\, y x^2,\, x^6,\, \dots\}
\end{equation}
The associated first and second differentials are given by
\begin{align}\label{DiffsC27}
&\rmd u = \begin{pmatrix} \rmd u_1 \\ \rmd u_3 \\ \rmd u_5 \end{pmatrix} 
= \begin{pmatrix} x^2 \\ x \\ 1  \end{pmatrix} 
\frac{\rmd x}{\partial_y f(x,y;\lambda)},\quad
\rmd r = \begin{pmatrix} \rmd r_1 \\ \rmd r_3 \\ \rmd r_5 \end{pmatrix} 
= \begin{pmatrix} \rho_{1}(x,y) \\ \rho_{3}(x,y) \\  \rho_{5}(x,y) \end{pmatrix} 
\frac{\rmd x}{\partial_y f(x,y;\lambda)}, \\
&\quad \rho_{1}(x,y) = x^3 ,\quad
\rho_{3}(x,y) = 3 x^4 + \lambda_4 x^2,\quad 
\rho_{5}(x,y) = 5 x^5  + 3 \lambda_4 x^3 + 2 \lambda_6 x^2 + \lambda_8 x. \notag
\end{align}
\end{27Curve}

Trigonal curves have canonical forms of two types: $(3,3\mFr{+}1)$, and $(3,3\mFr{+}2)$,
$\mFr \in \Natural$.

\begin{33m1Curve}
The canonical trigonal curve of genus $3\mFr$ is defined by 
\begin{equation}\label{33g1C}
f(x,y;\lambda) \equiv -y^3 + x^{3\mFr + 1} 
+ y \sum_{i=0}^{2\mFr}  \lambda_{3i+2} x^{2\mFr-i}
+ \sum_{i=1}^{3\mFr}  \lambda_{3i+3} x^{3\mFr-i}.
\end{equation}
The Weierstrass gap sequence is 
$\mathfrak{W} = \{3i-2 \mid i=1,\,\dots,\, \mFr\} \cup \{3i-1 \mid i=1,\,\dots,\, 2\mFr\}$,
sorted ascendingly.
Standard not normalized first kind differentials have the form
\begin{align}
\begin{split}
&\rmd u_{3i-2} = \frac{y x^{\mFr-i} \rmd x}{\partial_y f(x,y;\lambda)},\quad i=1,\,\dots,\,\mFr,\\
&\rmd u_{3i-1} = \frac{x^{2\mFr-i}\rmd x}{\partial_y f(x,y;\lambda)}\quad i=1,\,\dots,\,2\mFr.
\end{split}
\end{align}
\end{33m1Curve}
\begin{34Curve}
The simplest curve of this type is a $(3,4)$-curve defined by 
\begin{equation} \label{34C}
f(x,y;\lambda) \equiv -y^3 + x^4 + \lambda_2 y x^2 + \lambda_5 y x
+ \lambda_6 x^2 + \lambda_8 y + \lambda_9 x + \lambda_{12}, 
\end{equation}
with the gap sequence $\mathfrak{W}=\{1,2,5\}$, and the ordered list of monomials
\begin{equation}
\mathfrak{M} = \{1,\, x,\, y,\, x^2,\, y x,\, y^2,\, x^3,\, y x^2,\, y^2 x,\, \dots\}.
\end{equation}
The  system of associated first and second differentials consists of
\begin{align}\label{DiffsC34}
&\rmd u = \begin{pmatrix} \rmd u_1 \\ \rmd u_2 \\ \rmd u_5 \end{pmatrix} 
= \begin{pmatrix} y \\ x \\ 1  \end{pmatrix} 
\frac{\rmd x}{\partial_y f(x,y;\lambda)},\quad
\rmd r = \begin{pmatrix} \rmd r_1 \\ \rmd r_2 \\ \rmd r_5 \end{pmatrix} 
= \begin{pmatrix} \rho_{1}(x,y) \\ \rho_{2}(x,y) \\  \rho_{5}(x,y) \end{pmatrix} 
\frac{\rmd x}{\partial_y f(x,y;\lambda)}, \\
&\quad \rho_{1}(x,y) =x^2 ,\quad
\rho_{2}(x,y) =2 x y,\quad 
\rho_{5}(x,y) = 5x^2 y  + \tfrac{2}{3}\lambda_2^2 x^2 
+ \lambda_6 y + \tfrac{2}{3} \lambda_2 \lambda_5 x. \notag
\end{align}
\end{34Curve}

\begin{33m2Curve}
The canonical trigonal curve of genus $3\mFr+1$ is defined by 
\begin{equation}\label{33g2C}
f(x,y;\lambda) \equiv -y^3 + x^{3\mFr + 2} 
+ y \sum_{i=0}^{2\mFr+1}  \lambda_{3i+1} x^{2\mFr+1-i}
+ \sum_{i=1}^{3\mFr+1}  \lambda_{3i+3} x^{3\mFr+1-i}.
\end{equation}
The Weierstrass gap sequence is 
$\mathfrak{W} = \{3i-1 \mid i=1,\,\dots,\, \mFr\} \cup \{3i-2 \mid i=1,\,\dots,\, 2\mFr+1\}$,
sorted ascendingly.
Standard not normalized first kind differentials have the form
\begin{align}
\begin{split}
&\rmd u_{3i-1} = \frac{y x^{2\mFr+1-i} \rmd x}{\partial_y f(x,y;\lambda)},\quad i=1,\,\dots,\,2\mFr+1,\\
&\rmd u_{3i-2} = \frac{x^{\mFr-i}\rmd x}{\partial_y f(x,y;\lambda)}\quad i=1,\,\dots,\,\mFr.
\end{split}
\end{align}
\end{33m2Curve}

\begin{Remark}
In the hyperelliptic case, modality is $M=0$,  and on
trigonal curves  $M=\mFr-1$. 
\end{Remark}

\subsection{Abel map}
Let the Abel map $\mathcal{A}:  \mathcal{C} \to \Jac( \mathcal{C})$ 
be defined with respect to the not normalized differentials $\rmd u$:
\begin{gather}\label{AbelM}
 \mathcal{A}(P) = \int_{\infty}^P \rmd u,\qquad P=(x,y)\in \mathcal{C}.
\end{gather}
Recall, that infinity serves as the basepoint. The Abel map is also defined on  
any symmetric product $\mathcal{C}^n$. Namely, 
given a positive divisor $D = \sum_{i=1}^n P_i$,
we have
$$ \mathcal{A}(D) = \sum_{i=1}^n \mathcal{A}(P_i).$$

The Abel map is invertible on $\mathfrak{C}_g \subset \mathcal{C}^g$,
which consists of non-special divisors $D$ of degree $g$. 
The problem of finding $D$ from
\begin{equation}\label{uAMap}
u(D) = \sum_{i=1}^g \int_{\infty}^{P_i} \rmd u,\qquad D=\sum_{i=1}^g P_i \in \mathfrak{C}_g,
\end{equation}
 is known as the J\emph{acobi inversion problem} since  \cite{jac1835}.
A solution of the problem on hyperelliptic curves was known since the end of the 19-th century, see
 \cite[Chap.\,IX]{bakerAF}.

Normalized holomorphic differentials $\rmd v$,
and normalized period lattice $\{1_g, \tau\}$ are obtained as follows
\begin{gather}\label{NormJac}
\rmd v = \omega^{-1} \rmd u,\qquad\qquad
\tau = \omega^{-1}\omega'.
\end{gather}
The Riemann period matrix
$\tau$ belongs to the Siegel upper half-space $\mathfrak{S}_g$ of degree $g$, that is
$\tau$ is symmetric  with a positive imaginary part: $\tau^t=\tau$, $\Imn \tau >0$.

We also define the Abel map $\bar{\mathcal{A}}$ with respect to normalized differentials:
\begin{gather}\label{AbelMNorm}
 \bar{\mathcal{A}}(P) = \int_{\infty}^P \rmd v,\qquad P=(x,y)\in \mathcal{C}.
\end{gather}

In addition to the Abel map  $\mathcal{A}$, which produces first kind integrals, we introduce the map
$\mathcal{B}$, which produces second kind integrals:
\begin{gather}\label{SIntMap}
\begin{split}
& \mathcal{B}(P) = \int_{\infty}^P \rmd r,\qquad P=(x,y)\in \mathcal{C}, \\
& \mathcal{B}(D) = \sum_{i=1}^n \mathcal{B}(P_i),\qquad D \in \mathcal{C}^n.
\end{split}
\end{gather}
\begin{Remark}\label{R:RegC}
The integral in \eqref{SIntMap} requires regularization, since
the basepoint serves as the pole of second kind differentials. Such a regularization
is suggested in  \cite{BerLey2018}, in a way similar to regularization 
of the Weierstrass $\zeta$-function. The problem of regularization contains
 finding a constant vector, which depends of $\lambda$.
In \cite{BerLey2018} the regularization constants are obtained for $(3,4)$, $(3,5)$, $(3,7)$, $(4,5)$-curves.
In the hyperelliptic case these constants equal zero.
\end{Remark}

\subsection{Theta function}
In terms of normalized coordinates $v$, and the Riemann period matrix $\tau$
the Riemann \emph{theta function} $\theta: \Complex^g \times \mathfrak{S}_g \to \Complex$  is defined by
\begin{gather}\label{ThetaDef}
 \theta(v;\tau) = \sum_{n\in \Integer^g} \exp \big(\imath \pi n^t \tau n + 2\imath \pi n^t v\big).
\end{gather}
The definition works for any $\tau \in \mathfrak{S}_g$, though 
not all $\tau$ relates to a Jacobian variety if $g \geqslant 4$.
In what follows, we deal only with $\tau$ obtained from period matrices $\omega$ and $\omega'$ 
related to a curve $\mathcal{C}$.

Let a theta function with characteristic $[\varepsilon]= (\varepsilon', \varepsilon)^t$ be defined by
\begin{equation}\label{ThetaDefChar}
 \theta[\varepsilon](v;\tau) = \exp\big(\imath \pi  \varepsilon'{}^t \tau \varepsilon'
 + 2\imath \pi  (v+\varepsilon)^t (\varepsilon')\big)  \theta(v+ \varepsilon + \tau \varepsilon';\tau),
\end{equation}
where $\varepsilon$, and  $\varepsilon'$ are $g$-component vectors 
with real values within the interval $[0,1)$.  
Every point $u$ in the fundamental domain of $\Jac(\mathcal{C})$ 
can be represented by its characteristic $[\varepsilon]$, namely
\begin{equation*}
u[\varepsilon] =   \omega \varepsilon +  \omega' \varepsilon'.
\end{equation*}

\subsection{Sigma function}
In the theory presented below, the leading role belongs to the modular-invariant, entire function,
known as $\sigma$-function. 

A generalization of the Weierstrass $\sigma$-function was suggested by Klein: 
in \cite{klein1886, klein1888} to hyperelliptic curves,  in \cite{klein1890} to an arbitrary curve of genus $3$.
In \cite[\S\,189]{bakerAF}, \cite[p.\,25]{bakerMPF}, $\sigma$-function arises
as a modification of theta functions with characteristics. In genus two
a characteristic of $\sigma$-function was introduced
by choosing the corresponding pair of branch points $\{e_1,\,e_2\}$ as basepoints
 in the Abel map, see \cite[\S\S\,25, 26]{bakerMPF}, namely
\begin{equation}
u(P_1+P_2) = \int_{e_1}^{P_1} \rmd u + \int_{e_2}^{P_1} \rmd u.
\end{equation}
In \cite[Art.\,20-22]{bakerMPF} 
series for genus two $\sigma$-functions with odd and even characteristics
were derived. Finally, Baker defines the fundamental $\sigma$-function 
by fixing the basepoint at infinity, which is supposed to be one of branch points, 
see \cite[Art.\,27]{bakerMPF}. In what follows, we define $\sigma$-function after \cite[Eq.(2.3)]{belHKF1996}:
\begin{equation}\label{SigmaThetaRel}
\sigma(u) = C \exp\big({-}\tfrac{1}{2} u^t \varkappa u\big) \theta[K](\omega^{-1} u;  \omega^{-1} \omega'),
\end{equation}
where $\varkappa = \eta \omega^{-1}$ is a symmetric matrix constructed from $\eta$ and $\omega$
defined in \eqref{etaM} and \eqref{omegaM}, constant $C$ does not depend of $u$,
and $[K]$ denotes the characteristic of the vector $K$ of Riemann constants.

On the other hand, $\sigma$-function is an analytic series in 
coordinates $u \in \Jac(\mathcal{C})$ and parameters $\lambda$ of $\mathcal{C}$,
which arise as a solution of a system of heat equations. 
This definition generalizes the definition of the Weierstrass $\sigma$-function given in \cite{weier1894}.
Namely, the entire function $\sigma(u;g_2,g_3)$ related to the Weierstrass form of elliptic curves
 is defined by the system of partial differential equations
\begin{gather}
\mathfrak{Q}_0 \sigma = 0,\qquad \mathfrak{Q}_2 \sigma = 0,\\
\begin{split}
&\mathfrak{Q}_0 = 1 - u \partial_u  + 4 g_2  \partial_{g_2} + 6 g_3 \partial_{g_3},\\
&\mathfrak{Q}_2 = - \tfrac{1}{24} g_2 u^2  - \tfrac{1}{2} \partial^2_u  + 6 g_3  \partial_{g_2} 
+ \tfrac{1}{3} g_2^2 \partial_{g_3},
\end{split}
\end{gather}
with the initial condition $\sigma(u;0,0)=u$.
Here and in what follows,  $\partial_a \equiv \partial/\partial a$. 
Indices of the differential operators $\mathfrak{Q}_k$
indicate the Sato weight: $\wgt \mathfrak{Q}_k = k$. Note, that $\wgt g_2 = 4$, and $\wgt g_3 = 6$.
Evidently, $\mathfrak{Q}_0$ is the Euler operator, which shows weights of all arguments.
Operator $\mathfrak{Q}_2$ is of the second order with respect to $u$,
and of the first order with respect to parameters $g_2$, $g_3$ of the curve.

After introducing $(n,s)$-curves, and analysing the rational case in \cite{bel1999},
Buchstaber and Leykin have developed the theory of constructing series expansions of
$\sigma$-functions associated with $(n,s)$-curves, \cite{BL2002,BL2004,BL2008}.
They call it the theory of multi-variable $\sigma$-functions. 
Every $(n,s)$-curve, 
more precisely a bundle $\mathcal{E}(\Lambda,\pi,\mathcal{C})$ of curves defined by \eqref{nsCurve},
is equipped with a unique function $\sigma(u;\lambda)$. 
The operators $\mathfrak{Q}_k$, which annihilate $\sigma$-function,
are obtained as a lift of the Lie algebra of vector fields from the base $\Lambda$
to the fibre  bundle $\mathcal{E}\big(\Lambda,\varpi,\Jac(\mathcal{C})\big)$.

The $\sigma$-function associated with $\mathcal{C}$ 
arises as a unique solution of the system  $\{\mathfrak{Q}_k \sigma = 0\}$  of heat equations
with the initial condition $\sigma(u;0) = S_{\mathfrak{p}}(u)$, where $S_{\mathfrak{p}}$ denotes
the Schur--Weierstrass polynomial  in coordinates $u$, see \cite{bel1999,NakS2010}, generated by the partition 
\begin{equation}\label{SWPolyPart}
\mathfrak{p} = \{\mathfrak{p}_g,\, \mathfrak{p}_{g-1},\, \dots, \mathfrak{p}_2,\, \mathfrak{p}_1\} ,\qquad
\mathfrak{p}_i = \mathfrak{w}_i - (i-1).
\end{equation}
The latter is determined by the Weierstrass gap sequence $\mathfrak{W} = \{\mathfrak{w}_{1},\,\mathfrak{w}_{2},\,\dots,\, \mathfrak{w}_{g}\}$ on $\mathcal{C}$.
Note, that 
$ \sum_{i=1}^g \mathfrak{p}_i  = \tfrac{1}{24} (n^2-1)(s^2-1)$.

\begin{Property}\label{Pr:WgtSigma}
\cite{bel1999} The Sato weight of the $\sigma$-function associated with an $(n,s)$-curve is 
 \begin{equation}\label{WgtSigma}
 \wgt \sigma = - \tfrac{1}{24} (n^2-1)(s^2-1).
 \end{equation}
\end{Property}

The system of heat equations generated by $\mathfrak{Q}_k$ produces recurrence relations between coefficients 
of a series for the  $\sigma$-function in closed form.

\begin{Remark}
Note that a series of $\sigma(u;\lambda)$ is analytic with respect to all arguments: $u$ and $\lambda$, 
and so works for all $\lambda \in \Lambda$,
including strata with decreased genera.
\end{Remark}

In  \cite{EGOY2017} a detailed exposition of
the theory of Buchstaber and Leykin on the heat equations for multi-variable $\sigma$-functions 
with examples of $(2,3)$, $(2,5)$, $(2,7)$, $(3,4)$-curves can be found, with
recurrence relations derived. The power series expansion of the $\sigma$-function associated with
the most general form of an elliptic curve
$ y^2 = \lambda_1 y x + \lambda_3 y + x^3 + \lambda_2 x^2 + \lambda_4 x +\lambda_6$ is derived  in \cite{EilbOn2020}.

First  terms in the expansion of $\sigma$-function associated with a $(3,4)$-curve with extra terms 
can be found in \cite{EEMOP2007},  with a cyclic (superelliptic) $(3,5)$-curve  in \cite{BG2006},
with cyclic trigonal curves of genera six and seven in \cite{E2010}, and 
with a cyclic $(4,5)$-curve in \cite{EE2009}.

Relations between $\sigma$- and $\theta$-functions, on one hand, and the entire $\tau$-function, 
on the other hand, are explained in detail in \cite{EEG2010, NakST2010, HE2012}.
The theory of $\sigma$-functions associated with space curves is developed in 
\cite{MK2013,KMP2013} for  $(3,4,5)$-,  $(3,7,8)$-, and  $(6,13,14,15,16)$-curves, 
in \cite{ayano2014} for telescopic curves,  in \cite{KMP2019} for
more general trigonal cyclic curves. 
In \cite{KorShra2012} $\sigma$-function is defined for a compact Riemann surface 
which is not directly associated with an algebraic curve. In \cite{NakTT2016}
$\sigma$-functions for compact Riemann surfaces characterized by their Weierstrass semigroups
were derived from Sato's theory of the universal Grassmannian manifolds. 

\subsection{Vector of Riemann constants}
A formula for computing 
the vector of Riemann constants $K$ with respect to a given base point 
can be found in \cite[Eq.\,(2.4.14)]{Dub1981}.
In the hyperelliptic case, the vector is computed in \cite[p.\,14]{fay1973}, and  equals 
the sum of all odd characteristics of the fundamental set of $2g+1$ half-integer characteristics
which represent  branch points, according to \cite[\S\S\,200--202]{bakerAF}.

From \cite[Eq.\,(2.4.20)]{Dub1981} we know that 
$[K]$ is among half-integer characteristics, if $\mathcal{C}$ 
is not superelliptic.  
Property~\ref{Pr:WgtSigma} together with the relation \eqref{SigmaThetaRel}
between $\sigma$- and $\theta$-functions lead to the following
\begin{Theorem}
$\theta[K]$ as a function of not normalized coordinates has the maximal weighted order of vanishing equal to 
$\mathfrak{d}={-}\wgt \sigma$, and defined by \eqref{WgtSigma}, at $u=0$, that is
\begin{equation}
\forall \mathfrak{i} <\mathfrak{d}\quad   \partial_{u_1}^{\mathfrak{i}} \theta[K](0;\omega^{-1} \omega')=0,
\quad  \partial_{u_1}^{\mathfrak{d}} \theta[K](0;\omega^{-1} \omega')\neq 0.
\end{equation}
\end{Theorem}
This criteria allows to find the characteristic of vector of Riemann constants 
on non hyperelliptic curves, and singles out this characteristic  on hyperelliptic curves of even genera.
See examples in \\
\texttt{https://community.wolfram.com/groups/-/m/t/3296279}.

\subsection{Multiply periodic $\wp$-functions}
$\sigma$-Function generates abelian functions known as
multiply-periodic functions after \cite{bakerMPF}, or Kleinian $\wp$-functions after \cite{belHKF1996}, namely
\begin{gather}\label{WPdef}
\wp_{i,j}(u) = -\frac{\partial^2 \log \sigma(u) }{\partial u_i \partial u_j },\qquad
\wp_{i,j,k}(u) = -\frac{\partial^3 \log \sigma(u) }{\partial u_i \partial u_j \partial u_k},\quad \text{etc.}
\end{gather}
The notation arose in \cite[p.\,294]{bakerAF}.
In what follows, we  refer to them as $\wp$-functions.
All $\wp$-functions are periodic with respect to the period lattice $\{\omega, \omega'\}$.
Namely, let $m$, $m'$ be $g$-component vectors of integers, then
$$\wp_{i,j}(u + \omega m + \omega' m') = \wp_{i,j}(u).$$

\subsection{Divisor classes}
Traditionally, the Jacobian variety $\Jac(\mathcal{C})$ is described as
the group $\Pic^0$ of divisors of degree zero factored out by principal divisors.
Every class of  equivalent divisors on a curve of genus $g$
has a representative $D = \sum_{k=1}^l P_k - l \infty$, $0 \leqslant l \leqslant g$,  
such that the positive part contains no groups of points in involution\footnote{A group of points 
connected by involution on an $(n,s)$-curve
 consists of points $(a,b_i)$, $i=1$, \ldots, $n$ such that $y=b_i$ solve $f(a,y;\lambda)=0$.
A non-special divisor may contain $n-1$ points from a group in involution, but not all $n$,
since $\sum_{i=1}^n (a,b_i) \sim n \infty$.}.
We call such representatives \emph{reduced divisors}, and define by 
their positive parts: $D = \sum_{k=1}^l P_k$, 
since the poles are located at infinity, which serves as the basepoint.
We define the degree of $D$ as  $\deg D = l$.
In what follows, we always assume that a positive divisor $D$ contains no groups of points in involution.

Reduced divisors of degree less than $g$ are \emph{special}, and $\theta[K]$, as well as $\sigma$, vanishes 
on Abel images of such divisors, according to the Riemann vanishing theorem. 
Reduced divisors of degree $g$ are \emph{non-special}. 
Every non-special divisor represents its class uniquely. 

Let $\Sigma = \{ u \in \Jac(\mathcal{C}) \mid \sigma(u)=0 \}$. Up to the normalization \eqref{NormJac},
$\Sigma$ coincides with the theta divisor, see \cite[p.\,38]{Dub1981}.
\begin{Remark}
As follows from \eqref{WPdef}, 
$\wp$-functions are defined on $\Jac(\mathcal{C}) \backslash \Sigma$.
Therefore, the abelian function field $\mathfrak{A}(\mathcal{C})$ 
is defined over $\Jac(\mathcal{C}) \backslash \Sigma$.
At the same time, by taking proper limits and cancelling singularities, 
one can obtain fields of functions over strata of $\Sigma$
from  $\mathfrak{A}(\mathcal{C})$. 
\end{Remark}

Let $\mathfrak{C}_g$ be the subspace of $\mathcal{C}^g$ which consists of reduced divisors of degree $g$.
We also introduce  $\mathfrak{C}_{g-l} \subset \mathcal{C}^{g-l}$, $l=1$, \ldots, $g-1$.
Each $\mathfrak{C}_{g-l}$ consists of reduced divisors of the fixed degree equal to $g-l$.  
The Abel map \eqref{uAMap} on $\mathfrak{C}_{g-l}$ is obtained by taking the limit $P_{k} \to \infty$,
$k=g-l+1$, \ldots, $g$. That is $\mathfrak{C}_{g-l} + l \infty \subset \mathcal{C}^g$.
 
The Riemann vanishing theorem implies that 
$\mathfrak{C}_g = \{D \mid \sigma(\mathcal{A}(D)) \neq 0\}$. 
Subspaces  $\mathfrak{C}_{g-l}$ are also defined by means 
of the order of vanishing of $\sigma$-function, see \cite[Proposition\,5.7]{matsPrev2008},
Actually,
\begin{gather}\label{DivStrat}
\begin{split}
\mathfrak{C}_{g-l} = \{D \mid \sigma(\mathcal{A}(D)) = 0,\ 
\forall \mathfrak{i} < \mathfrak{r}= {\textstyle \sum_{i=1}^{l} \mathfrak{p}_i} \quad  
&\partial^{\mathfrak{i}}_{u_1} \sigma(\mathcal{A}(D)) = 0,\\
&\partial^{\mathfrak{r}}_{u_1}  \sigma(\mathcal{A}(D)) \neq 0\},\quad
l=1,\,\dots,\, g,
\end{split}
\end{gather}
where $\mathfrak{p}_i$ are defined in \eqref{SWPolyPart}.
Note that $\mathfrak{C}_0$ 
consists of the class of points equivalent to $g \infty \in \mathcal{C}^g$,
whose Abel image is the origin of $\Jac(\mathcal{C})$, namely, $\mathcal{A}({\mathfrak{C}_0})=0$.

\begin{Remark}
$\mathfrak{C}_{g-l}$, $l=0$, $1$, \ldots, $g$, are disjoint, and define stratification of $\Jac(\mathcal{C})$:
\begin{equation}
\Jac(\mathcal{C}) \backslash \Sigma = \mathcal{A}(\mathfrak{C}_{g}),\qquad 
\Sigma = \cup_{l=1}^g \mathcal{A}(\mathfrak{C}_{g-l}).
\end{equation}
\end{Remark}
Relations with the Wirtinger varieties $W_\ell$, see \cite[p.\,38]{EEMOP2008}, are given by 
$W_\ell = \cup_{l=0}^\ell \mathcal{A}(\mathfrak{C}_{l})$.

\subsection{Polynomial functions on a curve}
As we fixed the basepoint at infinity,  divisors on $\mathcal{C}$  are described by means of polynomial functions
from the ring $\Complex[x,y]/ f(x,y;\lambda)$, according to Abel's theorem. 
We denote by $\mathcal{R}_w$  a \emph{polynomial function of weight} $w$;
the weight shows the degree of the divisor of zeros $(\mathcal{R}_w)_0$.
And$ (\mathcal{R}_{\mathfrak{w}})_0$  is obtained from the system
\begin{equation}\label{Rzeros}
\mathcal{R}_{\mathfrak{w}}(x,y) = 0,\qquad f(x,y)=0. 
\end{equation}
Not all points of  $(\mathcal{R}_{\mathfrak{w}})_0$ can  be chosen arbitrarily.

\begin{Theorem}\label{T:PolyFunct}
A polynomial function $\mathcal{R}_{w}$ of weight $w\geqslant 2g$ 
is uniquely defined by a positive divisor $D$ of degree 
$w-g$ such that $D \subset (\mathcal{R}_{w})_0$.
 \end{Theorem}
\begin{proof}
A polynomial function $\mathcal{R}_{w}$  
is constructed from monomials $\{\mathfrak{m}_{\widetilde{w}} \mid 
\widetilde{w} \leqslant w \}$, namely
\begin{equation}\label{RwDef}
\mathcal{R}_{w}(x,y) = \mathfrak{m}_w + \sum_{\widetilde{w} < w}
c_{\widetilde{w}} \mathfrak{m}_{\widetilde{w}}.
 \end{equation}
 If $w\geqslant 2g$, there exist $w-g+1$ such monomials. 
Note that $\mathcal{R}_{w}$ is monic, that is the leading coefficient equals  $1$.
 Thus,  \eqref{RwDef} contains $w-g$ unknown coefficients,
 which are uniquely determined from $w-g$ points of $\mathcal{C}$. 
Indeed, let $D =\sum_{k=1}^{w-g} (x_k,y_k)$. If all points of $D$
are distinct, then
\begin{subequations}\label{DetR}
\begin{equation}\label{DetRD}
 \mathcal{R}_{w}(x,y; D) = 
 \frac{\small \begin{vmatrix}  
 \mathfrak{m}_{w}(x,y) & \mathfrak{m}_{w-1}(x,y) 
 & \dots & \mathfrak{m}_0(x,y)  \\
 \mathfrak{m}_{w}(x_1,y_1) 
 & \mathfrak{m}_{w-1}(x_1,y_1) & \dots &  \mathfrak{m}_0(x_1,y_1) \\
 \vdots & \ddots & \vdots\\
 \mathfrak{m}_{w}(x_{w-g},y_{w-g}) 
 & \mathfrak{m}_{w-1}(x_{w-g},y_{w-g}) & \dots &  
 \mathfrak{m}_0(x_{w-g},y_{w-g}) 
  \end{vmatrix}}
 {\small \begin{vmatrix}  
 \mathfrak{m}_{w-1}(x_1,y_1) & \dots &  \mathfrak{m}_0(x_1,y_1) \\
 \vdots & \ddots & \vdots\\
 \mathfrak{m}_{w-1}(x_{w-g},y_{w-g}) & \dots &  
 \mathfrak{m}_0(x_{w-g},y_{w-g}) \end{vmatrix}}.
\end{equation}
If some points coincide, say  $P_k =P_1$, $k=2$, \ldots $m$, then row $k+1$ in the numerator and 
row $k$ in the denominator of \eqref{DetRD} are replaced with 
\begin{equation}\label{DetRM}
\Big(\frac{\rmd^{k-1} }{\rmd x^{k-1}} \mathfrak{m}_{\widetilde{w}}(x,y(x)) 
\Big|_{\substack{x=x_1\\ y(x_1)=y_1} }\Big).
\end{equation}
\end{subequations}
The  formula \eqref{DetR} produces  $\mathcal{R}_{w}$ from 
$D_{w-g} \subset (\mathcal{R}_{w})_0$  uniquely. 
\end{proof}

\begin{Corollary}\label{C:AddP}
A polynomial function $\mathcal{R}_{w}$ of weight $w\geqslant 2g$,
constructed from a positive divisor $D$, $\deg D = w-g$,  
produces a complement  non-special divisor $D^\ast$, $\deg D^\ast = g$,
such that $(\mathcal{R}_{w})_0 = D + D^\ast$.
\end{Corollary}

\section{Abelian function fields associated with curves}\label{s:AbelFunct}
We denote by $\mathfrak{A}(\mathcal{C})$ the abelian function field associated 
with a curve $\mathcal{C}$ of genus $g$. We construct this  field from $\wp$-functions generated 
by  $\sigma$-function associated with the curve. 

\begin{Property}
$\mathfrak{A}(\mathcal{C})$ is a differential field, see \cite{BL2008}, which means it has the 
property: if $\Phi \in \mathfrak{A}(\mathcal{C})$, 
then $\partial_{u_{\mathfrak{w}_i}} \Phi \in \mathfrak{A}(\mathcal{C})$,
$i=1$, \ldots, $g$. The differentiation satisfies the additive rule and 
the Leibniz product rule. As a differential field, 
$\mathfrak{A}(\mathcal{C})$ has $g$ generators, see \cite{BL2008},
and Theorem~\ref{T:Gens}.
\end{Property}

\begin{Property}
$\mathfrak{A}(\mathcal{C})$ has the structure of a ring of polynomial (hyperelliptic case), or rational 
(non-hyperelliptic case) functions
in some basis $\wp$-functions. The ideal of the ring defines $\Jac(\mathcal{C}) \backslash \Sigma$.
\end{Property}

These properties are derived from the structure of abelian function fields associated with algebraic curves.
Such a function field serves to uniformize the curve. There are exist relations between $\wp$-functions, 
which we call identities.
In this section we recall how to solve the uniformization problem, 
known as the Jacobi inversion problem, 
and describe two techniques of obtaining identities for $\wp$-functions.

In what follows, we focus on hyperelliptic and trigonal curves,
which have been studied intensively, and broadly illustrated in the literature.

\subsection{Jacobi inversion problem}

Fixing the base point at infinity, we guarantee that  a solution of the Jacobi inversion problem
is expressed by means of polynomial functions from $\Complex[x,y] / f(x,y;\lambda)$, 
according to \cite[\S\,164]{bakerAF}.
We formulate the problem as follows.

\newtheorem*{JIP}{Jacobi inversion problem}
\begin{JIP}
Given a point $u\in  \Jac(\mathcal{C}) \backslash \Sigma$,
find a  divisor $D \in \mathfrak{C}_g$ such that $\mathcal{A}(D) = u$.
\end{JIP}

\begin{22g1Curve}
A solution of the Jacobi inversion problem 
on hyperelliptic curves was given in  \cite[\S\;216]{bakerAF},
see also \cite[Theorem 2.2]{belHKF1996}.
On a curve defined by \eqref{22g1C} 
the required divisor $D$ such that $u=\mathcal{A}(D)$
is obtained from  the system 
\begin{gather}\label{EnC22g1}
\begin{split}
&\mathcal{R}_{2g}(x;u) \equiv x^{g} -  \sum_{i=1}^{g} x^{g-i}  \wp_{1,2i-1}(u) = 0,\\ 
&\mathcal{R}_{2g+1}(x,y;u) \equiv 2 y + \sum_{i=1}^{g} x^{g-i}  \wp_{1,1,2i-1}(u) = 0.
\end{split}
\end{gather}
In other words, $D$ is the common divisor of zeros of the two polynomial functions $\mathcal{R}_{2g}$,
$\mathcal{R}_{2g+1}$ of weights $2g$, and $2g+1$, correspondingly.
\end{22g1Curve}

A solution of the Jacobi inversion problem on trigonal curves is derived from the Klein formula in \cite{bel2000}.
Based on the Klein formula for superelliptic $(n,s)$-curves, expressions for $\wp_{1,\mathfrak{w}_i}$, $i=1$, \ldots, $g$,
in terms of coordinates of $D$ are obtained in \cite{matsPrev2008},
as well as similar relations in the case of special divisors $D$. In other words, 
the first equation from the system which gives a solution to the Jacobi inversion problem
is suggested. 

The Jacobi inversion problem on special divisors,
 in the case of hyperelliptic curves, is addressed in \cite{EHKKLS2012};
stratification of $\Sigma$ by means of the order of vanishing of $\sigma$-function,
and inversion of integrals on the curve are suggested. 
Stratification of Jacobian varieties of superelliptic curves is discussed in \cite{matsPrev2014};
strata are indexed by Young-diagrams, and described by means of
the order of vanishing of $\sigma$-function.


A complete solution of the Jacobi inversion problem on $(n,s)$-curves
is obtained in \cite{BLJIP22} by a more elegant technique based on the residue theorem, see Eq.\;\eqref{rExpr}.
Solution of the Jacobi inversion problem on trigonal, tetragonal and pentagonal curves 
are elaborated as illustration. Below, the solutions on canonical trigonal curves from \cite{BLJIP22} are presented.

\begin{33m1Curve}
On a canonical $(3,3\mFr +1)$-curve, defined by \eqref{33g1C},
the pre-image $D$ of $u = \mathcal{A}(D)$ is given by the system
\begin{gather}\label{EnC33m1}
\begin{split}
&\mathcal{R}_{6\mFr+1}(x,y;u) \equiv x^{2\mFr} 
-  y \sum_{i=1}^{\mFr} \wp_{1,3i-2}(u) x^{\mFr-i}
-  \sum_{i=1}^{2\mFr} \wp_{1,3i-1}(u) x^{2\mFr-i} = 0,\\ 
&\mathcal{R}_{6\mFr+2}(x,y;u) \equiv 2 y x^{\mFr} 
- y \sum_{i=1}^{\mFr} \big(\wp_{2,3i-2}(u) - \wp_{1,1,3i-2}(u) \big) x^{\mFr-i} \\
&\qquad\qquad\qquad\qquad\qquad 
- \sum_{i=1}^{2\mFr} \big(\wp_{2,3i-1}(u) - \wp_{1,1,3i-1}(u) \big) x^{2\mFr-i}= 0. 
\end{split}
\end{gather}
\end{33m1Curve}
\begin{33m2Curve}
On a canonical $(3,3\mFr +2)$-curve, defined by \eqref{33g2C}, the required $D$ is given by the system
\begin{gather}\label{EnC33m2}
\begin{split}
&\mathcal{R}_{6\mFr+2}(x,y;u) \equiv y x^{\mFr} 
-  y \sum_{i=1}^{\mFr} \wp_{1,3i-1}(u) x^{\mFr-i} 
-  \sum_{i=1}^{2\mFr+1} \wp_{1,3i-2}(u) x^{2\mFr+1-i} = 0,\\ 
&\mathcal{R}_{6\mFr+3}(x,y;u) \equiv 2 x^{2\mFr+1} 
- y \sum_{i=1}^{\mFr} \big( \wp_{2,3i-1}(u)  - \wp_{1,1,3i-1}(u) \big) x^{\mFr-i} \\
&\qquad\qquad\qquad\qquad\qquad 
- \sum_{i=1}^{2\mFr+1} \big(\wp_{2,3i-2}(u) - \wp_{1,1,3i-2}(u) \big) x^{2\mFr+1-i}= 0. 
\end{split}
\end{gather}
\end{33m2Curve}

\begin{34Curve} 
A solution of the Jacobi inversion problem on $\mathfrak{C}_3 = \Jac(\mathcal{C})\backslash \Sigma$
 is given by the system
\begin{gather}\label{JIPC34}
\begin{split}
&\mathcal{R}_6(x,y;u) = x^2 - y \wp_{1,1}(u) - x \wp_{1,2}(u) - \wp_{1,5}(u), \\
&\mathcal{R}_7(x,y;u) = 2 x y - y \big(\wp_{1,2}(u) - \wp_{1,1,1}(u)\big) 
+ x \big(\wp_{2,2}(u) - \wp_{1,1,2}(u)\big) \\
&\phantom{\mathcal{R}_7(x,y;u) = 2 x y y} - \big(\wp_{2,5}(u) - \wp_{1,1,5}(u) \big).  
\end{split}
\end{gather}
On the other hand,  polynomial functions of weights $6$ and $7$ are constructed 
from  a divisor $D = \sum_{k=1}^3 (x_k,y_k)$ by the  formula \eqref{DetR}, namely (with all distinct points)
\begin{gather}\label{DetRC34}
\mathcal{R}_6(x,y;D) = 
\frac{
\small\begin{vmatrix} 
x^2 & y & x & 1 \\ 
x_1^2 & y_1 & x_1 & 1 \\
x_2^2 & y_2 & x_2 & 1 \\
x_3^2 & y_3 & x_3 & 1
\end{vmatrix}}{
\small\begin{vmatrix} 
y_1 & x_1 & 1 \\
y_2 & x_2 & 1 \\
y_3 & x_3 & 1
\end{vmatrix}}, \qquad
\mathcal{R}_7(x,y;D) = 2 
\frac{
\small\begin{vmatrix} 
x y & y & x & 1 \\ 
x_1 y_1 & y_1 & x_1 & 1 \\
x_2 y_2 & y_2 & x_2 & 1 \\ 
x_3 y_3 & y_3 & x_3 & 1
\end{vmatrix}}{
\small\begin{vmatrix} 
y_1 & x_1 & 1 \\
y_2 & x_2 & 1 \\
y_3 & x_3 & 1
\end{vmatrix}},
\end{gather}

Assuming, that \eqref{JIPC34} and \eqref{DetRC34} define the same functions,
expressions for 
\begin{gather*}
\wp_{1,1}(u),\quad \wp_{1,2}(u),\quad \wp_{1,5}(u),\quad \wp_{1,1,1}(u) - \wp_{1,2}(u),\\ 
\wp_{1,1,2}(u) - \wp_{2,2}(u),\quad \wp_{1,1,5}(u) - \wp_{2,5}(u),
\end{gather*}
which we call basis functions, in terms of coordinates of $D$ can be found.

According to \eqref{DivStrat}, with $\mathfrak{p}=\{3,1,1\}$, special divisors are stratified as follows
\begin{gather}\label{StratC34}
\begin{split}
&\mathfrak{C}_2 = \{D \mid \sigma(\mathcal{A}(D))=0,\, \partial_{u_1} \sigma(\mathcal{A}(D)) \neq 0\},\\
&\mathfrak{C}_1 = \{D \mid \sigma(\mathcal{A}(D))=0,\, \partial_{u_1} \sigma(\mathcal{A}(D)) = 0,
\partial_{u_2} \sigma(\mathcal{A}(D)) \neq 0\},\\
&\mathfrak{C}_0 = \{D \mid \sigma(\mathcal{A}(D))=0,\, \forall \mathfrak{i} <5\ \, 
\partial_{u_1}^{\mathfrak{i}} \sigma(\mathcal{A}(D)) = 0,
\partial_{u_5} \sigma(\mathcal{A}(u)) \neq 0\}.
\end{split}
\end{gather}
Since $\sigma$ vanishes on $\Sigma$,  \eqref{JIPC34} is replaced with\begin{gather}\label{RSpec}
 \sigma(u)^2 \mathcal{R}_6(x,y;u) = 0,\qquad\quad
 \sigma(u)^3 \mathcal{R}_7(x,y;u) = 0,
\end{gather}
where $\wp$-functions are expressed in terms of derivatives of $\sigma$,
and cancel all vanishing terms. As a result, the both $\mathcal{R}_6$,
$\mathcal{R}_7$ reduce to the same equation:
\begin{align*}
&D\in \mathfrak{C}_2,\ \ u=\mathcal{A}(D) & &y \sigma_1(u) + x \sigma_2(u) + \sigma_5(u) = 0,\\
&D\in \mathfrak{C}_1,\ \ u=\mathcal{A}(D) & & x \sigma_2(u) + \sigma_5(u) = 0,
\end{align*}
where $\sigma_{\mathfrak{w}_i}(u) \equiv \partial_{u_{\mathfrak{w}_i}} \sigma(u)$.
This particular case agrees with (\cite{matsPrev2008}, Theorem 5.1), including part (3).
\end{34Curve}

On the contrary, in the case of $(3,7)$-curve, we have $\mathfrak{W} = \{1,2,4,5,8,11\}$, and $\mathfrak{p} = \{6,4,2,2,1,1\}$.
Thus, part (3) fails for $\mathfrak{C}_2$ since $\sigma_5(u)=0$ occurs in the denominator, and for $\mathfrak{C}_1$
due to the denominator $\sigma_8(u)=0$.

\begin{27Curve} 
A solution of the Jacobi inversion problem on $\mathfrak{C}_3$ is given by the system
\begin{subequations}\label{JIPC27}
\begin{align}
&\mathcal{R}_6(x;u) = x^3 - x^2 \wp_{1,1}(u) - x \wp_{1,3}(u) - \wp_{1,5}(u), &\\
&\mathcal{R}_7(x,y;u) = 2 y + x^2  \wp_{1,1,1}(u) + x \wp_{1,1,2}(u) + \wp_{1,1,5}(u). &   
\end{align}
\end{subequations}

According to \eqref{DivStrat}, with $\mathfrak{p}=\{3,2,1\}$, special divisors are stratified as follows
\begin{gather}\label{StratC27}
\begin{split}
&\mathfrak{C}_2 = \{D \mid \sigma(\mathcal{A}(D))=0,\, \partial_{u_1} \sigma(\mathcal{A}(D)) \neq 0\},\\
&\mathfrak{C}_1 = \{D \mid \sigma(\mathcal{A}(D))=0,\, \, \forall \mathfrak{i} <3\ \, 
\partial_{u_1}^{\mathfrak{i}} \sigma(\mathcal{A}(D)) = 0,
\partial_{u_3} \sigma(\mathcal{A}(D)) \neq 0\}, \\
&\mathfrak{C}_0 = \{D \mid \sigma(\mathcal{A}(D))=0,\, \, \forall \mathfrak{i} <6\ \, 
\partial_{u_1}^{\mathfrak{i}} \sigma(\mathcal{A}(D)) = 0,
\partial_{u_1}^{6} \sigma(\mathcal{A}(D)) \neq 0\}.
\end{split}
\end{gather}
Since $\sigma$ vanishes on $\Sigma$, we use \eqref{RSpec}, and obtain, cf.\,\cite[Theorem\,5.1]{matsPrev2008},
\begin{align*}
&D\in \mathfrak{C}_2,\ \ u=\mathcal{A}(D) & &x^2 \sigma_1(u) + x \sigma_3(u) + \sigma_5(u) = 0,\\
&D\in \mathfrak{C}_1,\ \ u=\mathcal{A}(D) & & x \sigma_3(u) + \sigma_5(u) = 0.
\end{align*}
\end{27Curve}

In the case of $(2,9)$-curve, we have $\mathfrak{W} = \{1,3,5,7\}$, and $\mathfrak{p}=\{4,3,2,1\}$.
Then  $\forall \mathfrak{i} \,{<}\,6$ $\partial_{u_1}^{\mathfrak{i}} \sigma(u)=0$ on $\mathcal{A}(\mathfrak{C}_2)$, 
and $\forall \mathfrak{i} < 10$ $\partial_{u_1}^{\mathfrak{i}} \sigma(u)=0$ on $\mathcal{A}(\mathfrak{C}_1)$,
that agrees with  (\cite{matsPrev2014}, Proposition~5.7), but contradicts 
 (\cite{matsPrev2008}, Theorem 5.1(3)).

\subsection{Basis functions}
A solution of the Jacobi inversion problem indicates the abelian functions
which serve as a convenient choice of basis  in  $\mathfrak{A}(\mathcal{C})$. 

\begin{Definition}
We call the following  functions basis $\wp$-functions in the abelian function field $\mathfrak{A}(\mathcal{C})$ 
associated with a curve $\mathcal{C}$ with the Weietstrass gap sequence 
$\mathfrak{W}=\{\mathfrak{w}_i \mid i=1,\,\dots,\,g\}$:
\begin{subequations}\label{BasisF}
\begin{align}
&p_{\mathfrak{w}_i+1} = \wp_{1,\mathfrak{w}_i},&
&q_{\mathfrak{w}_i+2} = \wp_{1,1,\mathfrak{w}_i} &
&\text{if } \mathcal{C} \text{ is hyperelliptic};& \\
&p_{\mathfrak{w}_i+1} = \wp_{1,\mathfrak{w}_i},&
&q_{\mathfrak{w}_i+2} = \wp_{1,1,\mathfrak{w}_i} - \wp_{2,\mathfrak{w}_i}&
&\text{if } \mathcal{C} \text{ is trigonal}.& 
\end{align}
\end{subequations}
\end{Definition}

Coefficients of $x^k$, $k=0$, \ldots $g-1$, in $\mathcal{R}_{2g}$, $\mathcal{R}_{2g+1}$, without representation
in terms of $\wp$-functions, were used
as coordinates of hyperelliptic $\Jac(\mathcal{C}) \backslash \Sigma$, see
 \cite[Chap.\,IIIa, \S\,1]{mumfordII}, and so named \emph{Mumford coordinates}.
Mumford's construction of a hyperelliptic Jacobian variety implies, that
the affine ring of $\Jac(\mathcal{C}) \backslash \Sigma$ is isomorphic to 
$\Complex[p_{\mathfrak{w}_i+1},q_{\mathfrak{w}_i+2}]$, see \cite[Theorem 2.8]{uchida2011}.
In other words, we have
\begin{Theorem}\label{T:BasisFunctHE}
The abelian function field $\mathfrak{A}(\mathcal{C})$
associated with a hyperelliptic curve $\mathcal{C}$ 
is a polynomial ring in  basis $\wp$-functions.
In other words, 
every meromorphic function on $\Jac(\mathcal{C}) \backslash \Sigma$ 
is represented as a polynomial in the basis $\wp$-functions.
\end{Theorem}
Now, we add 
\begin{Conjecture}\label{C:BasisFunct}
The abelian function field $\mathfrak{A}(\mathcal{C})$
associated with a non-hyperelliptic curve $\mathcal{C}$ 
is a ring of rational functions in basis $\wp$-functions.
In other words, 
every meromorphic function on $\Jac(\mathcal{C}) \backslash \Sigma$ 
is represented as a rational function in the basis $\wp$-functions.
\end{Conjecture}

Basis $\wp$-functions in $\mathfrak{A}(\mathcal{C})$  provide uniformization of $\mathcal{C}$:
there exists a one-to-one correspondence between the basis functions and 
coordinates of divisors in $\mathfrak{C}_g$. 
Examples of related computations can be found in \\
\texttt{https://community.wolfram.com/groups/-/m/t/3243472}\\
\texttt{https://community.wolfram.com/groups/-/m/t/3252458}\\

\begin{Remark}
The ring of polynomial functions denoted by $\Complex[x,y] / f(x,y;\lambda)$
is defined, in fact, over $\mathfrak{A}(\mathcal{C})$, since all coefficients of 
$\mathcal{R}_w$, $w\geqslant 2g$, are expressible in terms of $\wp$-functions
on $\Jac(\mathcal{C})\backslash \Sigma$.
At the same time, if $w \geqslant 3g$, a divisor $D$, $\deg D = w-g$, which defines $\mathcal{R}_w$
can be considered 
as a sum of $\mathfrak{k} = [w/g]-1$ reduced divisors,
and  $\mathcal{R}_{w}$ is defined not only on 
$\Jac(\mathcal{C})\backslash \Sigma$ 
but also on $(\Jac(\mathcal{C})\backslash \Sigma )^\mathfrak{k}$. 
\end{Remark}

A solution of the Jacobi inversion problem implies
\begin{Theorem}\label{T:PolyFIdent}
Polynomial functions from $\Complex[x,y]/ f(x,y;\lambda)$
over $\Jac(\mathcal{C})\backslash \Sigma$ form a vector space
which we denote by $\mathfrak{P}(\mathcal{C})$.
Monomials $\upsilon_{\mathfrak{w}_i}$, $i=1$, \ldots, $g$, from \eqref{DiffNot}
serve as a basis  in $\mathfrak{P}(\mathcal{C})$.
\end{Theorem}
\begin{proof}
Indeed, polynomial functions of weights $2g$, $2g+1$, \ldots, $2g+n-1$
which give a solution of the Jacobi inversion problem, see \cite[Theorem 1]{BLJIP22}, serve to reduce 
monomials of these weights to linear combinations of the basis monomials.
Any polynomial function of weight higher than $2g+n-1$ 
admits a reduction by means of 
$\mathcal{R}_{2g}$,  $\mathcal{R}_{2g+1}$, \ldots, $\mathcal{R}_{2g+n-1}$. 
\end{proof}

\begin{Theorem}\label{T:IdentGen}
Polynomial functions $\mathcal{R}_w$, $w \geqslant 2g+n$, 
serve as generators of identities for $\wp$-functions on $\Jac(\mathcal{C})\backslash \Sigma$.
\end{Theorem}
\begin{proof}
Any polynomial function $\mathcal{R}_w$, $w \geqslant 2g+n$, is reduced to 
a function $\mathcal{R}^{\text{red}}_w$ of weight less than $2g$, according to Theorem~\ref{T:PolyFIdent}. 
Coefficients of $\mathcal{R}^{\text{red}}_w$ are expressed rationally in  basis $\wp$-functions.
At the same time, the function $\mathcal{R}_w$
with coefficients computed at $u \in \Jac(\mathcal{C}) \backslash \Sigma$
has a divisor of zeros of degree at least $2g$, since $(\mathcal{R}_w)_0$
contains  $\mathcal{A}^{-1}(u) \in \mathfrak{C}_g$ and a complement degree $g$ divisor,
as follows from Theorem~\ref{T:PolyFunct}. This contradiction means, that $\mathcal{R}^{\text{red}}_w$ vanishes, 
and coefficients of basis monomials 
serve as identities for $\wp$-functions.
\end{proof}

Theorem~\ref{T:IdentGen} provides a way of obtaining algebraic identities for $\Phi \in \mathfrak{A}(\mathcal{C})$.
There are exist several techniques of constructing polynomial functions for generating the identities.
Below we consider two of them: the ones based on the Klein formula, and
on the residue theorem.

\subsection{The Klein formula technique}\label{ss:KleinF}
Following \cite[p.\;45]{bakerMPF}, \cite[\S\,1.2.4]{belHKF1996}, we define
\begin{Definition}
The fundamental bi-differential of the second kind $\rmd \mathrm{B}(x,y;\tilde{x},\tilde{y})$
is symmetric: $\rmd \mathrm{B}(x,y;\tilde{x},\tilde{y}) = \rmd \mathrm{B}(\tilde{x},\tilde{y};x,y)$,
and has the only pole of the second order along the diagonal $x=\tilde{x}$.
\end{Definition}
According to \cite[Eq.\,(1.8)]{belHKF1996} (hyperelliptic case), and \cite[Definition 3.4]{EEL2000},
the fundamental bi-differential $\rmd \mathrm{B}$  has the form
\begin{equation}
\rmd \mathrm{B}(x,y;\tilde{x},\tilde{y}) = 
\frac{\mathcal{F}(x,y;\tilde{x},\tilde{y}) \rmd x \rmd \tilde{x}}
{(x-\tilde{x})^2 \big(\partial_y f(x,y;\lambda)\big) \big(\partial_{\tilde{y}} f(\tilde{x},\tilde{y};\lambda)\big)},
\end{equation}
where $\mathcal{F}: \mathcal{C}^2 \to \Complex$ is a polynomial of its variables,
symmetric in $(x,y)$ and $(\tilde{x},\tilde{y})$.

In the case of a hyperelliptic curve, see \cite[Eq.]{bakerAF}, \cite[Eq.\,(1.7)]{belHKF1996}, 
\begin{equation}\label{BiPolarHE}
\mathcal{F}(x,y;\tilde{x},\tilde{y}) = 2 y \tilde{y} + x^{g} \tilde{x}^g (x+\tilde{x}) + \sum_{i=1}^g x^{g-i} \tilde{x}^{g-i}
\Big(\lambda_{4i} (x+\tilde{x}) + 2 \lambda_{4i+2}\Big).
\end{equation}
A formula for $\mathcal{F}$ associated with canonical trigonal curves is given in \cite[Sect.\,4]{bel2000}.
A method of constructing $\mathcal{F}$ in the non-hyperelliptic case is suggested in \cite[\S\,3.2]{EEL2000}, 
\cite[Lemma 2.4]{EEMOP2007}.
The polynomial $\mathcal{F}$ was computed explicitly for
 a $(3,4)$-curve with extra terms in \cite[Eqs.\,(A.3), (A.4)]{EEMOP2007},
 a cyclic $(3,5)$-curve in \cite[Theorem 4.1]{BEGO2008}, cyclic trigonal curves 
of genera six and seven in \cite[Eqs.\,(10), and \S\,7.1]{E2010},
and a cyclic $(4,5)$-curve in \cite[Eq.\,(15)]{EE2009}.

Let  $u= \mathcal{A}(D)$ with $D\in \mathfrak{C}_g$, and $(x,y)\in D$.
Let  $(\tilde{x},\tilde{y})$ be an arbitrary point of $\mathcal{C}$.
The Klein formula has the form, see  \cite[\S\,217]{bakerAF},  
\cite[p.\,138]{bakerDE1903}, \cite[Theorem 3.4]{EEL2000},
\begin{equation}\label{KleinF}
\sum_{i,j=1}^g \wp_{\mathfrak{w}_i,\mathfrak{w}_j} \big(u - \mathcal{A}(\tilde{x},\tilde{y})\big) 
\upsilon_{\mathfrak{w}_i}(x,y) \upsilon_{\mathfrak{w}_j}(\tilde{x},\tilde{y}) \\
= \frac{\mathcal{F}(x,y;\tilde{x},\tilde{y})}{(x-\tilde{x})^2}.
\end{equation}
Multiplying \eqref{KleinF} by $(x-\tilde{x})^2$, and sending $(\tilde{x},\tilde{y})$ to infinity
by introducing the local parameter $\xi$ by \eqref{param}, 
we obtain polynomial functions $\mathcal{R}_{w}$ as coefficients in the expansion. 
This technique produces polynomial functions of all weights starting from $2g$ and greater.
Functions of the lowest weights give a solution of the
Jacobi inversion problem, in particular   $\mathcal{R}_{2g}$, and $\mathcal{R}_{2g+1}$ on
hyperelliptic and trigonal curves.

\subsection{The residue theorem technique}
Another technique, proposed in \cite{BLJIP22}, allows to derive 
identities for  $\wp$-functions 
from the curve equation and a system of associated first and second differentials.
We start with the formula \cite[Eq.\,(15)]{BLJIP22}
\begin{equation}\label{rExpr}
 \sum_{k=1}^g r_{\mathfrak{w}_i} (x_k,\,y_k)  = - \res\limits_{\xi=0} \bigg(r_{\mathfrak{w}_i}(\xi) \frac{\rmd}{\rmd \xi}
 \log \sigma\big(u- \mathcal{A}(\xi)\big) \bigg) \equiv \mathcal{B}_{\mathfrak{w}_i} (u),\quad i=1,\dots,g,
\end{equation}
which is, in fact, an application of the residue theorem. The formula \eqref{rExpr}
 contains series expansions about infinity of the second kind differentials $r_{\mathfrak{w}_i} (\xi)$,
and antiderivatives of the first kind differentials $\mathcal{A} (\xi)$. These expansions are derived
with the help of  \eqref{param} obtained from the curve equation \eqref{nsCurve}.

By means of \eqref{rExpr}, second kind integrals $\mathcal{B}(u)$, 
corresponding to the chosen differentials of the second kind $\rmd r$, are computed on
 $D \in \mathfrak{C}_g$  such that $u= \mathcal{A}(D)$.
These differentials have the form
\begin{equation}
\mathcal{B}_{\mathfrak{w}_i} (u) = - \zeta_{\mathfrak{w}_i} (u) + \text{abelian function}(u) + c(\lambda),\quad i=1,\,\dots,\, g,
\end{equation}
where $\zeta_{\mathfrak{w}_i} = \partial \log \sigma(u) / \partial u_{\mathfrak{w}_i}$ serves as 
a generalization of the Weierstrass $\zeta$-function,
and $c(\lambda)$ is a regularization constant vector, see Remark~\ref{R:RegC}.
\begin{Property}\label{P:DBDu}
The abelian functions in $\mathcal{B}_{\mathfrak{w}_j}$ are derived by differentiating 
$\mathcal{B}_{\mathfrak{w}_i}$, $i=1$, \ldots, $j-1$,
and expressed terms of $\wp_{1,\mathfrak{w}_i}$ and their derivatives.
\end{Property}

Polynomial functions $\mathcal{R}_{2g-1+\mathfrak{w}_i}$  with $D =\sum_{k=1}^g (x_k,y_k)
\subset (\mathcal{R}_{2g-1+\mathfrak{w}_i})_0$ 
are obtained by taking derivative of \eqref{rExpr} with respect to any $x_k$, 
\begin{equation*}
\frac{\rmd r_{\mathfrak{w}_i} (x_k,y_k)}{\rmd x_k}  
= \big(\partial_u \mathcal{B}_{\mathfrak{w}_i} (u)  \big)^t \frac{\rmd u(x_k,y_k)}{\rmd x_k}, \quad i=1,\,\dots,\, g.
\end{equation*}
Multiplying these equalities by $\partial_{y_k} f(x_k,y_k;\lambda)$, we obtain $g$ polynomial functions:
\begin{equation}\label{GenRFunct}
\mathcal{R}_{2g-1+\mathfrak{w}_i}(x,y;u) = \rho_{\mathfrak{w}_i}(x,y) 
- \sum_{j=1}^g  \upsilon_{\mathfrak{w}_j} (x,y) \partial_{u_{\mathfrak{w}_j}} \mathcal{B}_{\mathfrak{w}_i} (u).
\end{equation}
We call them \emph{main polynomial functions}. 
These functions are sufficient  to derive all identities for $\wp$-functions in $\mathfrak{A}(\mathcal{C})$.

\begin{Remark}
Main polynomial functions with $i=1$, \ldots, $n-1$
give a solution of the Jacobi inversion problem, and so define basis $\wp$-functions in $\mathfrak{A}(\mathcal{C})$. 
Note, in the hyperelliptic case, only one function $\mathcal{R}_{2g}$ is obtained. This reflects the fact
that only this function is essential, and $\mathcal{R}_{2g+1}$ is derived from $\mathcal{R}_{2g}$,
namely $\mathcal{R}_{2g+1}(x,y;u) = \partial_{u_1} \mathcal{R}_{2g}(x,y;u)$;
\end{Remark}

As indicated in \cite{BL2008},
$\mathfrak{A}(\mathcal{C})$ has $g$ generators.
\begin{Theorem}\label{T:Gens}
The basis functions $p_{\mathfrak{w}_i+1} = \wp_{1,\mathfrak{w}_i}$, $i=1$, \ldots, $g$, serve as generators in 
the differential field $\mathfrak{A}(\mathcal{C})$.
\end{Theorem}
\begin{proof}
It is sufficient to show that $\wp_{\mathfrak{w}_i,\mathfrak{w}_j}$,
$i,\,j=2$, \ldots, $g$, are expressible in terms of $\wp_{1,\mathfrak{w}_i}$, $i=1$, \ldots, $g$, and their
derivatives. All required relations follow from the main polynomial functions $\mathcal{R}_{2g-1+\mathfrak{w}_i}$.
Indeed, as seen from \eqref{GenRFunct}, funcitons $\wp_{\mathfrak{w}_i,\mathfrak{w}_j}$,
$i,\,j=2$, \ldots, $g$, arise as first terms in $\partial_{u_{\mathfrak{w}_j}} \mathcal{B}_{\mathfrak{w}_i} (u)$.
Due to Property~\ref{P:DBDu},  
the required expressions through $\wp_{1,\mathfrak{w}_i}$ and derivatives can be obtained.
\end{proof}

\subsection{Examples}
\begin{27Curve}
On the curve defined by \eqref{27C}, with first and second differentials as in \eqref{DiffsC27},
the formula \eqref{rExpr} produces the second kind integrals:
\begin{gather}
\begin{split}
&\mathcal{B}_{1}(u) =  - \zeta_1(u),\\
&\mathcal{B}_{3}(u) =  - \zeta_3(u) + \tfrac{1}{2} \wp_{1,1,1}(u),\\
&\mathcal{B}_{5}(u) =  - \zeta_5(u) + \tfrac{5}{6} \wp_{1,1,3}(u) + \tfrac{1}{24} \wp_{1,1,1,1,1}(u).
\phantom{mmmmmmmmmmm}
\end{split}
\end{gather}
According to  Property~\ref{P:DBDu}, we have
\begin{gather}
\begin{split}
&\mathcal{B}_{3}(u) =  - \zeta_3(u) + \tfrac{1}{2} \partial_{u_1}^2 \mathcal{B}_{1}(u),\\
&\mathcal{B}_{5}(u) =  - \zeta_5(u) + \tfrac{5}{6} \partial_{u_1}^2 \mathcal{B}_{3}(u)
- \tfrac{3}{8} \partial_{u_1}^4  \mathcal{B}_{1}(u).\phantom{mmmmmmmmmmmmm}
\end{split}
\end{gather}
The first  function produced by \eqref{GenRFunct}
is  $\mathcal{R}_{6}$. Then one can derive $\mathcal{R}_{7}(x,y;u) = \partial_{u_1} \mathcal{R}_{6}(x,y;u)$,
or apply \eqref{rExpr} to the second kind differential $\rmd r_2 = 2y$.
From $\mathcal{B}_{3}$, $\mathcal{B}_{5}$ we obtain
\begin{align*}
&\mathcal{R}_{8}(x,y;u) = 3 x^4  - x^2 \big(\tfrac{1}{2} \wp_{1,1,1,1}+ \wp_{1,3} - \lambda_4\big)
- x \big(\tfrac{1}{2} \wp_{1,1,1,3} + \wp_{3,3}\big) \\
&\qquad\qquad\qquad -  \big(\tfrac{1}{2} \wp_{1,1,1,5} + \wp_{3,5}\big),\\
&\mathcal{R}_{10}(x,y;u) = 5 x^5  + 3 \lambda_4 x^3 
- x^2 \big(\tfrac{1}{24} \wp_{1,1,1,1,1,1} + \tfrac{5}{6} \wp_{1,1,1,3} + \wp_{1,5} - 2 \lambda_6 \big) \\
&\qquad\qquad\qquad
- x \big(\tfrac{1}{24} \wp_{1,1,1,1,1,3} + \tfrac{5}{6} \wp_{1,1,3,3} + \wp_{3,5} - \lambda_8 \big) \\
&\qquad\qquad\qquad
- \big(\tfrac{1}{24} \wp_{1,1,1,1,1,5} + \tfrac{5}{6} \wp_{1,1,3,5} + \wp_{5,5}\big). 
\end{align*}
By reducing these polynomial functions to the basis monomials $1$, $x$, $x^2$, we find
\begin{align*}
&\wp_{3,3} = -\tfrac{1}{2} \wp_{1,1,1,3} + 3 \wp_{1,1} \wp_{1,3} + 3 \wp_{1,5},\\
&\wp_{3,5} = - \tfrac{1}{2} \wp_{1,1,1,5} + 3 \wp_{1,1} \wp_{1,5},\\ 
&\wp_{5,5} = - \tfrac{1}{24} \wp_{1,1,1,1,1,5} - \tfrac{5}{6} \wp_{1,1,3,5}
+ 5 \wp_{1,1}^2 \wp_{1,5} + 5 \wp_{1,3} \wp_{1,5} + 3 \lambda_4 \wp_{1,5}.
\end{align*}
For brevity, we omit the argument of $\wp$-functions.
\end{27Curve}

\begin{34Curve}
On the curve defined by \eqref{34C} with first and second differentials of the form \eqref{DiffsC34},
the formula \eqref{rExpr} produces second kind integrals:
\begin{gather}
\begin{split}
&\mathcal{B}_{1}(u) =  - \zeta_1(u),\\
&\mathcal{B}_{2}(u) =  - \zeta_2(u) - \wp_{1,1}(u),\\
&\mathcal{B}_{5}(u) =  - \zeta_5(u) + \tfrac{5}{12} \lambda_2 \wp_{1,2}(u) + \tfrac{5}{8} \wp_{1,2,2}(u) 
- \tfrac{5}{12} \wp_{1,1,1,2}(u) +\tfrac{1}{24} \wp_{1,1,1,1,1}(u);
\end{split}
\end{gather}
and according to  Property~\ref{P:DBDu}, we have
\begin{gather}
\begin{split}
&\mathcal{B}_{2}(u) =  - \zeta_2(u) - \partial_{u_1} \mathcal{B}_{1}(u),\\
&\mathcal{B}_{5}(u) =  - \zeta_5(u) + \big(\tfrac{5}{8} \partial_{u_2} 
+ \tfrac{5}{24} \partial_{u_1}^2  + \tfrac{5}{12}  \lambda_2 \big) \partial_{u_1} \mathcal{B}_{2}(u)
\phantom{mmmmmmmmmmmm}  \\
&\phantom{\mathcal{B}_{5}(u) =\ } 
+ \big( \tfrac{1}{4}  \partial_{u_1}^2 +  \tfrac{5}{12} \lambda_2\big) \partial^2_{u_1} \mathcal{B}_{1}(u).
\end{split}
\end{gather}
The first two functions obtained by \eqref{GenRFunct} from $\mathcal{B}_{1}$, $\mathcal{B}_{2}$ 
are  $\mathcal{R}_{6}$, $\mathcal{R}_{7}$, which produce expressions for $\wp_{2,2}$ and $\wp_{2,5}$.
Then from $\mathcal{B}_{5}$ we obtain $\mathcal{R}_{10}$, which produces an expression for $\wp_{5,5}$.
\end{34Curve}

\subsection{Identities for $\wp$-functions: Hyperelliptic case}
Identities for $\wp$-functions in the hyperelliptic case are thoroughly elaborated.
In \cite{bakerMPF}, genus two curves are considered in detail, and 
expressions for  $4$-index $\wp$-functions can be found in \cite[p.\;47--48]{bakerMPF}.
In genus three, expressions for  $4$-index $\wp$-functions are obtained in \cite{bakerDE1903}.
In \cite{belHKF1996}, the latter results are extended to arbitrary hyperelliptic curves,
and the focus has shifted from $4$-index $\wp$-functions to expressions for
$\wp_{1,1,\mathfrak{w}_i}(u) \wp_{1,1,\mathfrak{w}_j}(u)$ called \emph{fundamental cubic relations},
since they generalize the well known differential equation 
for the Weierstrass $\wp$-function, known as the Weierstrass cubic:
$$ \wp'(u)^2 = 4 \wp(u)^3 - g_2 \wp(u) - g_3.$$ 

Analyzing identities for  $\wp$-functions from \cite{belHKF1996}, we conclude
\begin{Theorem} 
Let $\mathcal{C}$ be a hyperelliptic curve in the canonical form \eqref{22g1C}  with
the Weierstrass gap sequence $\mathfrak{W} = \{  \mathfrak{w}_i = 2i-1 \mid i=1,\,\dots,\, g\}$. 
Then in $\mathfrak{A}(\mathcal{C})$ the following  hold:
\begin{itemize}
\item every $4$-index function $\wp_{i,j,k,l}$ is represented as a 
polynomial in $\wp_{i,j}$ with coefficients in $\Integer[\lambda]$;

\item every product $\wp_{1,1,\mathfrak{w}_i} \wp_{1,1,\mathfrak{w}_j}$ is represented
as a polynomial in $\wp_{i,j}$ with coefficients in $\Integer[\lambda]$.
\end{itemize}
\end{Theorem}
As we see below, fundamental cubic relations 
produce an algebraic model of $\Jac(\mathcal{C})\backslash \Sigma$.
And so called quartic relations, induced by
$$\big(\wp_{1,1,\mathfrak{w}_i} \wp_{1,1,\mathfrak{w}_j}\big)
\big(\wp_{1,1,\mathfrak{w}_k} \wp_{1,1,\mathfrak{w}_l}\big) - 
\big(\wp_{1,1,\mathfrak{w}_i} \wp_{1,1,\mathfrak{w}_k}\big)
\big(\wp_{1,1,\mathfrak{w}_j} \wp_{1,1,\mathfrak{w}_l}\big) =0,$$
produce a  model of the Kummer variety $\Kum(\mathcal{C}) = \Jac(\mathcal{C}) / \pm$.

A convenient  matrix form of the fundamental cubic relations is suggested in \cite[p.\,39]{bakerMPF}
for genus $2$ curves. Extension to hyperelliptic curves is given in
 \cite[\S\,3]{belHKF1996}, and elaborated in more detail in \cite{ath2008, ath2011}.
The matrix form originates in the Klein formula \eqref{KleinF}.
Let $(g{+}2)\times(g{+}2)$ symmetric matrices $\mathsf{P}$ and $\mathsf{L}$ be defined by
\begin{gather}\label{PHmatr}
\begin{split}
&(x-\tilde{x})^2 \sum_{i,j=1}^g x^{g-i} \tilde{x}^{g-j} \wp_{2i-1,2j-1} (u) 
  = \mathbf{x}^t \mathsf{P} \tilde{\mathbf{x}}, \\
&x^{g} \tilde{x}^g (x+\tilde{x}) + \sum_{i=1}^g x^{g-i} \tilde{x}^{g-i}
\Big(\lambda_{4i} (x+\tilde{x}) + 2 \lambda_{4i+2}\Big) = \mathbf{x}^t \mathsf{L} \tilde{\mathbf{x}},
\end{split}
\end{gather}
where $\mathbf{x}=(1,\,x,\,x^2,\,\dots,\, x^{g+1})^t$, and  
$\tilde{\mathbf{x}} =(1,\,\tilde{x},\,\tilde{x}^2,\,\dots,\, \tilde{x}^{g+1})^t$. 
The Klein formula implies, cf. \cite[Eq.\,(4.2)]{ath2008},
\begin{equation}\label{KleinFMatr}
 \mathbf{x}^t  \mathsf{H} \tilde{\mathbf{x}} + 2 y \tilde{y} = 0,\qquad \mathsf{H} = \mathsf{P} - \mathsf{L}.
\end{equation}
Then  $y$, and $x^g$, $x^{g+1}$ are reduced by means of $\mathcal{R}_{2g+1}$,
and $\mathcal{R}_{2g}$, namely
\begin{gather}\label{WPCols}
x^g = \Upsilon_1^t \mathbf{b},\qquad y = \Upsilon_2^t \mathbf{b},\qquad x^{g+1} = \Upsilon_3^t \mathbf{b},\\
\intertext{where $\mathbf{b} = (1,\,x,\,x^2,\,\dots,\, x^{g-1})^t$, and }
\Upsilon_1 = \begin{pmatrix} \wp_{1,\mathfrak{w}_g} \\ \vdots \\  \wp_{1,\mathfrak{w}_2} \\ \wp_{1,1} \end{pmatrix},\quad
\Upsilon_1^\circ =  \begin{pmatrix}  0 \\ \wp_{1,\mathfrak{w}_g} \\ \vdots \\  \wp_{1,\mathfrak{w}_2} \end{pmatrix}, \quad
\Upsilon_2 = {-}\frac{1}{2} \begin{pmatrix}   \wp_{1,1,\mathfrak{w}_g} \\ \vdots \\
\wp_{1,1,\mathfrak{w}_2} \\ \wp_{1,1,1} \end{pmatrix},\quad
\Upsilon_3 = \wp_{1,1} \Upsilon_1 + \Upsilon_1^\circ.
\end{gather}
 Similarly,
$\tilde{y}$,  $\tilde{x}^g$,  $\tilde{x}^{g+1}$ are reduced to 
$\widetilde{\mathbf{b}} = (1,\,\tilde{x},\,\tilde{x}^2,\,\dots,\, \tilde{x}^{g-1})^t$.
Let $\mathsf{T} = (1_g, \Upsilon_1,\Upsilon_3)^t$, then the reduction is written as 
$\mathbf{x} = \mathsf{T} \mathbf{b}$;
and the fundamental cubic relations are obtained in the matrix form
\begin{equation}\label{JacEqCHE}
\mathsf{T}^t \mathsf{H} \mathsf{T} + 2 \Upsilon_2 \Upsilon_2^t = 0.
\end{equation}
Note, that $\Upsilon_2 \Upsilon_2^t$ is a  tensor product,
and $2\Upsilon_2 \Upsilon_2^t = (\tfrac{1}{2} \wp_{1,1,\mathfrak{w}_i} \wp_{1,1,\mathfrak{w}_j})$.

\begin{Theorem} \cite[Theorem 3.3]{belHKF1996} 
We have 
$\rank \mathsf{H} = 3$ at points of $\Jac(\mathcal{C}) \backslash \Sigma$ but not half-periods, 
$\rank \mathsf{H} = 2$ at half-periods, $\rank \sigma(u)^2 \mathsf{H} = 3$
at points of $\mathcal{A}(\mathfrak{C}_{g-1})$.
\end{Theorem}

A matrix representation of fundamental cubic relations associated with hyperellitic curves in the form $(2,2g+2)$
up to genus three are given in \cite{ath2008}. Expressions for $4$-index $\wp$-functions 
and lists of fundamental cubic relations are also obtained.

A covariant approach, 
based on elementary representation theory, is developed in  \cite{AEE2003, AEE2004, ath2008, ath2011}.
Families of hyperelliptic curves related by $\mathfrak{sl}(2,\Complex)$ transformations are considered.
Let $\mathcal{H}^1(\mathcal{C})$ denote the vector space of holomorphic $1$-forms, which  
is a $g$-dimensional irreducible $\mathfrak{sl}(2,\Complex)$-module.
$n$-Index $\wp$-functions belong to $\Sym^{\otimes n} \mathcal{H}^1(\mathcal{C})$
which is decomposable into irreducible components, each associated with a 
highest weight element. Similarly, each identity between $\wp$-functions 
falls into some irreducible component, and knowledge of the highest weight element
is enough to determine all identities in that module. This approach `has a clear 
calculation advantage, and also reveals the structure of the equations'.
Further development of the equivariant approach is given in
\cite{ath2012}.

\subsection{Identities for $\wp$-functions:  Non-hyperelliptic case }
Enormous strides have been made in listing identities for non-hyperellitic curves of genera up to seven.
By the Klein formula technique
expressions for all $4$-index $\wp$-function, and 
all identities linear and quadratic in $3$-index $\wp$-functions are derived.
A $(3,4)$-curve with extra terms $\lambda_1 x y^2$,
$\lambda_4 y^2$, $\lambda_3 x^3$ is elaborated in \cite{EEMOP2007}.
Identities for $\wp$-functions associated with a cyclic (or superellipitic) $(3,5)$-curve are obtained in \cite{BEGO2008, E2011}.
A comparison of the identities associated with two types of genus three curves:  $(2,7)$, and cyclic $(3,4)$
is given in \cite{EEO2011}. Identities associated with a cyclic $(4,5)$-curve are presented in  \cite{EE2009},
and associated with cyclic $(3,7)$ and $(3,8)$-curves in \cite{E2010}.

In the next section we will return to identities for $\wp$-functions,
and illustrate the residue theorem technique in obtaining such identities associated with a $(3,4)$-curve.

\section{Algebraic models of Jacobian varieteis}

\subsection{Hyperelliptic case}
In \cite[Chap.\,IIIa \S\,1]{mumfordII} an algebraic construction of  hyperelliptic $\Jac(\mathcal{C})\backslash \Sigma$
is suggested. Though this construction does not involve $\wp$-functions, it proves
\begin{Theorem}\label{T:JacHE} \cite[Theorem 2.9]{uchida2011}
The homomorphism 
\begin{gather}
\begin{split}
\Complex [p_{2i}, q_{2i+1}] / 
\langle \J_{4+2g}, \J_{6+2g}, \dots, \J_{2+4g} \rangle & \to \Complex[\wp_{1,2i-1}, \wp_{1,1,2i-1}]\\
p_{2i}& \mapsto \wp_{1,2i-1},\\
q_{2i+1}& \mapsto \wp_{1,1,2i-1}
\end{split}
\end{gather}
is well defined and an isomorphism. In particular, the ideal $\langle \dots \rangle$ is formed by the  equations
\begin{equation}\label{JacEqHE}
\J_{2+2g + 2i}\big(p_{2}, p_4, \dots, p_{2g}, q_{3}, q_{5}, \dots, q_{2g+1}\big) = 0.
\end{equation}
\end{Theorem}
\begin{Remark}
One should take into account, that polynomials in $p_{2i}$, $q_{2i+1}$,
denoted by $\Complex [p_{2i}, q_{2i+1}]$ above, are considered over the field $\Complex$, 
where coefficients of the equation of $\mathcal{C}$ are supposed to belong. 
Working with $(n,s)$-curves, we consider $\lambda$ as parameters in $\Lambda = \Complex^{\mathfrak{N}}$.
The homomorphism is defined on polynomials in $p_{2i}$, $q_{2i+1}$ 
with coefficients from  $\Integer[\lambda]$.
\end{Remark}

Explicit equations \eqref{JacEqHE}, which give a model of hyperelliptic $\Jac(\mathcal{C})\backslash \Sigma$, 
can be obtained from the fundamental cubic relations \eqref{JacEqCHE} 
by eliminating $\wp_{2i-1,2j-1}$ with $i,j=2$, \dots, $g$. Indeed,
$\tfrac{1}{2}g(g+1)$ equations in \eqref{JacEqCHE} contain all $2$-index $\wp$-functions,
of number $\tfrac{1}{2}g(g+1)$, as seen from the definition of matrix $\mathsf{P}$.
Eliminating all but basis functions, we obtain the required $g$ equations \eqref{JacEqHE}.
This proves 
\begin{Theorem} \cite[Corollary 3.2.1]{belHKF1996}
The map $\varphi: \Jac(\mathcal{C})\backslash \Sigma \to \Complex^{g+g(g+1)/2}$,
$$\varphi(u) = \big(\{\wp_{1,1,2i-1}(u) \mid i=1,\dots, g\}, \{\wp_{2i-1,2j-1}(u) \mid i=1,\dots, g, j=i,\dots, g\}\big)$$
is a meromorphic embedding. The range of $\varphi$ is the intersection of  $g(g+1)/2$ cubics
induced by the fundamental cubic relations.
\end{Theorem}

\subsection{Trigonal case}
\begin{Theorem} \cite[Corollary 3.3]{bel2000}
The variety $\Jac(\mathcal{C})\backslash \Sigma$ of a trigonal curve of genus $g$ 
can be realized as an algebraic variety in $\Complex^{4g+\delta}$,
where $\delta = 2(g-3[g/3])$. This subvariety is defined as the set of common zeros of a system 
of $3g+\delta$ polynomials generated by 
the functions $\mathcal{R}_{2g}$, $\mathcal{R}_{2g+1}$ which give a solution to the Jacobi inversion problem.
\end{Theorem}

Now, we explain how to obtain such an algebraic model of $\Jac(\mathcal{C})\backslash \Sigma$ by
employing the method from  \cite[\S\,2]{bel2000}, which works for both hyperelliptic and trigonal curves.

In the cases of $n=2$ and $n=3$, the Jacobi inversion problem  is solved by 
means of two functions $\mathcal{R}_{2g}$, $\mathcal{R}_{2g+1}$ 
of weights $2g$, and $2g+1$, respectively. 
Let the common divisor of zeros of these two functions at $u\in \Jac(\mathcal{C})$ 
be  $D$ such that $u = \mathcal{A}(D)$. We write this fact down as
$\big(\mathcal{R}_{2g}(u),\mathcal{R}_{2g+1}(u)\big)_0 = D$.
At the same time, $(\mathcal{R}_{2g})_0 = D + D^\ast$, where $D^\ast$
is the  the inverse of $D$ induced by the map $u \mapsto {-} u$, since 
$\mathcal{R}_{2g}$ is even with respect to $u$.
That is, $D^\ast = \big(\mathcal{R}_{2g}(u),\mathcal{R}_{2g+1}(-u)\big)_0$.
Let $\big(\mathcal{R}_{2g+1}(u)\big)_0 = D + \widetilde{D}$, 
and $\big(\mathcal{R}_{2g+1}(-u)\big)_0 = D^\ast + \widetilde{D}^\ast$.
Then
$$ \mathcal{M}(x,y) = \frac{\mathcal{R}_{2g+1}(x,y;u) \mathcal{R}_{2g+1}(x,y;-u) }{\mathcal{R}_{2g}(x,y;u)}$$
is a polynomial function of weight $2g+2$, and $(\mathcal{M})_0 = \widetilde{D} + \widetilde{D}^\ast$.
Let $\mathcal{M}$ be a linear combination of the first $g+3$ monomials, with weights up to $2g+2$, 
from the ordered list~$\mathfrak{M}$. Then the polynomial function
\begin{equation}\label{R2g1Sq}
\mathcal{Q}(x,y;u) \equiv \mathcal{R}_{2g+1}(x,y;u) \mathcal{R}_{2g+1}(x,y;-u) - \mathcal{M}(x,y) \mathcal{R}_{2g}(x,y;u)
\end{equation}
contains terms with $y$ to the power equal or greater than $n$. These terms are eliminated 
with the help of the curve equation. According to \eqref{R2g1Sq},  $\mathcal{Q}$ has the weight $4g+2$. 
At the same time, after cancelation of higher weight terms 
the actual weight becomes less than $4g+2$, and so we have
\begin{equation*}
\mathcal{Q}(x,y;u) -  \mathcal{N}(x) f(x,y;\lambda) \equiv 0,
\end{equation*}
where $\wgt \mathcal{N} = (n-2)(s-2)$.

\begin{Theorem}\label{JRels}
On curves with $n=2$ or $3$ the equality 
\begin{equation}\label{R2g1SqEq}
\mathcal{R}_{2g+1}(x,y;u) \mathcal{R}_{2g+1}(x,y;-u) - \mathcal{M}(x,y) \mathcal{R}_{2g}(x,y;u)
 -  \mathcal{N}(x) f(x,y;\lambda) \equiv 0.
\end{equation}
 produces  $g$ identities 
 for basis $\wp$-functions $p_{\mathfrak{w}_i+1}$, $q_{\mathfrak{w}_i+2}$,
 $i=1$, \ldots, $g$, which we denote by $\J_{\delta + \mathfrak{w}_i}$,  $i=1$, \ldots, $g$,
 where $\delta + \mathfrak{w}_i$ shows the weight, 
 $\delta = 2g+3$ in the hyperelliptic case, and 
 $\delta = \mathfrak{w}_g + 6 + \mathfrak{w}_i$ in the trigonal case.
\end{Theorem}
\begin{proof}
First, we find unknown coefficients of $\mathcal{M}$ and $\mathcal{N}$ expressed in terms of basis $\wp$-functions, and parameters $\lambda$. Since we know the weights: $\wgt \mathcal{M} = 2g+2$, and $\wgt \mathcal{N} = (n-2)(s-2)$,
the number of unknowns: $\# \mathcal{M}$ in $\mathcal{M}$,
and  $\# \mathcal{N}$ in $\mathcal{N}$ is easily counted, as well as the number of equations $\# \text{Eqs}$, 
which arise as coefficients of monomials.
 By direct computations on each family of curves, taking into account that $\wp_{i,j}(-u) = \wp_{i,j}(u)$, 
 $\wp_{i,j,k}(-u) = -\wp_{i,j,k}(u)$, we find

\begin{tabular}{ccccc}
& $\wgt \mathcal{N}$ & $\# \mathcal{N}$ & $\# \mathcal{M}$ & $\# \text{Eqs.}$ \\
 $(2,2g+1)$ & $0$ &$1$ & $g+3$ & $2g+4$ \\
 $(3,3\mFr + 1)$ & $3\mFr - 1$ & $\mFr$ & $3\mFr +3$ & $10\mFr+2$ \\
 $(3,3\mFr + 2)$ & $3\mFr$ & $\mFr+1$ & $3\mFr +4$ & $10\mFr+5$ 
\end{tabular}

\smallskip
\noindent
Thus, after elimination of all unknown coefficients of $\mathcal{M}$ and $\mathcal{N}$, we obtain
\begin{itemize}
\item  $g$ equations on a  $(2,2g+1)$-curve;
\item  $6\mFr -1$ equations on a  $(3,3\mFr+1)$-curve, which contain $\wp_{1,\mathfrak{w}_i}$,
$\wp_{1,1,\mathfrak{w}_i}$ with $i=1$, \ldots, $g=3\mFr$, and
$\wp_{2,\mathfrak{w}_i}$ with  $i=2$, \ldots, $g$;
\item $6\mFr + 2$ equations on a  $(3,3\mFr+2)$-curve, which contain $\wp_{1,\mathfrak{w}_i}$,
$\wp_{1,1,\mathfrak{w}_i}$ with $i=1$, \ldots, $g=3\mFr+1$, and
$\wp_{2,\mathfrak{w}_i}$ with  $i=2$, \ldots, $g$.
\end{itemize}

On a hyperelliptic curve $\mathcal{C}$ the obtained $g$ equations
 give a model of $\Jac(\mathcal{C})\backslash \Sigma$. 
 
On a trigonal curve $\mathcal{C}$  we obtain $2g{-}1$ equations,
 which produce expressions for $\wp_{2,\mathfrak{w}_i}$,  $i=2$, \ldots, $g$,
in terms of the basis $\wp$-functions, and $g$ equations which give a model of $\Jac(\mathcal{C})\backslash \Sigma$.
\end{proof}

\begin{Remark}
The identities $\J_{\delta + \mathfrak{w}_i} = 0$,  $i=1$, \ldots, $g$, obtained in Theorem~\ref{JRels},
give an algebraic model of $\Jac(\mathcal{C})\backslash \Sigma$. 
In the hyperelliptic case,
those are the differential equations \eqref{JacEqHE}
mentioned in Theorem~\ref{T:JacHE}.
These identities form the ideal of the ring mentioned in
Theorem~\ref{T:BasisFunctHE}, and Conjecture~\ref{C:BasisFunct}.
\end{Remark}
In the case of a trigonal curve we have 
\begin{Corollary}\label{T:JacTrig}
The homomorphism 
\begin{gather}
\begin{split}
\Integer [\lambda] (p_{\mathfrak{w}_i+1}, q_{\mathfrak{w}_i+2}) / 
\langle \J_{\mathfrak{w}_g + 7}, \J_{ \mathfrak{w}_g + 6 + \mathfrak{w}_2}, \dots,
 \J_{2\mathfrak{w}_g + 6} \rangle & \to \mathfrak{A}(\mathcal{C}) \\
p_{\mathfrak{w}_i+1}& \mapsto \wp_{1,\mathfrak{w}_i},\\
q_{\mathfrak{w}_i+2}& \mapsto \wp_{1,1,\mathfrak{w}_i} - \wp_{2,\mathfrak{w}_i}
\end{split}
\end{gather}
is well defined and an isomorphism.In particular, the ideal $\langle \dots \rangle$ is formed by the  equations
\begin{equation}\label{JacEq}
\J_{\mathfrak{w}_g + 6 + \mathfrak{w}_i}\big(p_{2}, p_{\mathfrak{w}_2+1}, \dots, p_{\mathfrak{w}_g+1}, 
q_{3}, q_{\mathfrak{w}_2+2}, \dots, q_{\mathfrak{w}_g+2}\big) = 0,
\end{equation}
which are obtained from Theorem~\ref{JRels}.
\end{Corollary}

\subsection{Examples}
\begin{34Curve}
Equations obtained from \eqref{R2g1SqEq} have the form
\begin{align}
&\mathcal{G}_6 \equiv {-}\wp_{1,1,1}^2 + 4 \wp_{1,1}^3 - 4 \wp_{1,1} \wp_{2,2}  + 4 \wp_{1,5} 
+ \wp_{1,2}^2 + 4 \lambda_2 \wp_{1,1}^2 = 0, \notag\\
&\mathcal{G}_7 \equiv {-}2\wp_{1,1,1} \wp_{1,1,2} + 8 \wp_{1,1}^2 \wp_{1,2} - 2 \wp_{1,2} \wp_{2,2}
- 4 \wp_{2,5} \notag \\
&\phantom{\mathcal{G}_7 \equiv\ } + 4 \lambda_2 \wp_{1,1} \wp_{1,2} + 4 \lambda_5 \wp_{1,1} = 0,\notag \\
&\mathcal{G}_{10} \equiv {-}2\wp_{1,1,1} \wp_{1,1,5} + 8 \wp_{1,1}^2 \wp_{1,5} 
+ 2 \wp_{1,2} \wp_{2,5} - 4 \wp_{1,5} \wp_{2,2} + 4 \lambda_2 \wp_{1,1} \wp_{1,5}  \notag \\
&\phantom{mmma} + 4 \lambda_8 \wp_{1,1} 
+ \wp_{1,1} \big({-}\wp_{1,1,2}^2 + 4 \wp_{1,1} (\wp_{1,2}^2 + \lambda_6)
+ \wp_{2,2}^2 \big)=0,\label{GenCubRelsC34} \\
&\mathcal{G}_{11} \equiv  {-}2\wp_{1,1,2} \wp_{1,1,5} + 8 \wp_{1,1} \wp_{1,2} \wp_{1,5} 
+ 2 \wp_{2,2} \wp_{2,5} + 4 \lambda_9 \wp_{1,1} \notag\\
&\phantom{mmma} + \wp_{1,2} \big({-}\wp_{1,1,2}^2 + 4 \wp_{1,1} (\wp_{1,2}^2 + \lambda_6)
+ \wp_{2,2}^2 \big)=0,\notag \\
&\mathcal{G}_{14} \equiv  {-} \wp_{1,1,5}^2 + 4 \wp_{1,1} \wp_{1,5}^2  + \wp_{2,5}^2 + 4 \lambda_{12} \wp_{1,1} \notag \\
&\phantom{mmma} + \wp_{1,5} \big({-}\wp_{1,1,2}^2 + 4 \wp_{1,1} (\wp_{1,2}^2 + \lambda_6)
+ \wp_{2,2}^2 \big)=0. \notag
\end{align}
From the first two equations, which are linear in $\wp_{2,2}$, $\wp_{2,5}$,  we find
\begin{gather}\label{wp22wp25C34}
\begin{split}
&\wp_{2,2} = {-} p_2^{-1} \big(\tfrac{1}{4} q_3 (q_3 + 2 p_3) - p_6 \big) + p_2 (p_2+\lambda_2),\\
&\wp_{2,5} = {-} \tfrac{1}{2} (q_3 + p_3) q_4  - \tfrac{1}{2}  p_2 (p_2+\lambda_2)  q_3
+ p_2  (p_2 p_3 + \lambda_5)   \\
&\phantom{\wp_{2,5} =} +  \tfrac{1}{2} p_2^{-1} \big(\tfrac{1}{4} q_3 (q_3 + 2 p_3) - p_6  \big) (q_3 + 2 p_3). 
\end{split}
\end{gather}
By eliminating \eqref{wp22wp25C34} from the remaining three equations of \eqref{GenCubRelsC34},
we obtain an algebraic model of $\Jac(\mathcal{C})\backslash \Sigma$ of a $(3,4)$-curve $\mathcal{C}$, namely
\begin{align}
&\J_{12} \equiv {-} 2 p_2 (q_3+p_3) (q_7 - q_3 q_4) - p_2^2 q_4^2 
-  q_3^2 \big(\tfrac{1}{2} p_2 q_4 + (\tfrac{1}{2} q_3 + p_3)^2  - p_6 \big) \notag \\
& \quad - (2 p_2 q_4 - q_3^2) (p_6 + p_2^2 (p_2+\lambda_2)) 
+ 2 q_3 \big(2 p_3 p_6 - p_2^2 (p_2 p_3 + \lambda_5) \big)\notag \\
&\quad - 4 p_6^2 + 4 p_2^3 (p_6 + p_3^2 + \lambda_6) + \lambda_8 p_2^2 = 0, \notag \\
&\J_{13} \equiv {-} (2 q_7 - q_3 q_4) \big(p_2 q_4 - \tfrac{1}{4} q_3 (q_3+ 2 p_3) 
 + p_6 + p_2^2 (p_2+\lambda_2) \big) 
 \label{JacRelsC34}\\
&\quad - 2 p_2^2 q_4 \big(p_3 (p_2 + \lambda_2) + p_2 p_3 + \lambda_5 \big)   \notag \\
&\quad
+ 4 p_2^2 \big( p_3 (p_6 + p_3^2 + \lambda_6) + p_3 p_6 + \lambda_9)\big) = 0,\notag \\
&\J_{16} \equiv {-} p_2 (q_7^2 + p_6 q_4^2 - (q_3+p_3) q_7 q_4) 
- (q_7 q_3 -2 p_6 q_4) \times \notag \\
&\quad \times \big(\tfrac{1}{4} q_3^2 - p_6 -p_2^2 (p_2+\lambda_2) \big) 
 - p_3 q_3 (q_7 q_3 - p_6 q_4) - q_7 \big((p_3 q_3 - 2 p_6)p_3 \notag \\
&\quad + 2p_2^2 (p_2p_3+\lambda_5) \big) 
+ 4p_2^2 \big(p_6 (p_6+p_3^2+\lambda_6) + \lambda_{12} \big) = 0. \notag
\end{align}
The equations are written in terms of the basis $\wp$-functions in $\mathfrak{A}(\mathcal{C})$:
\begin{gather}\label{BasisWPC34}
\begin{split}
&p_2 = \wp_{1,1},\phantom{mmmmmt} p_3 = \wp_{1,2},\qquad p_6 = \wp_{1,5},\\
&q_3 = \wp_{1,1,1} - \wp_{1,2},\quad q_4 = \wp_{1,1,2} - \wp_{2,2},\quad  q_7 = \wp_{1,1,5} - \wp_{2,5}.
\end{split}
\end{gather}
\end{34Curve}

\begin{27Curve}
From \eqref{R2g1SqEq} we obtain the following equations, which serve as
an algebraic model of $\Jac(\mathcal{C}) \backslash \Sigma$
of a $(2,7)$-curve $\mathcal{C}$:
\begin{align}
&\J_{10} \equiv {-}2 q_3 q_7 - q_5^2 - 2 p_2 q_3 q_5  - (p_4 + p_2^2) q_3^2 
+12 p_2^2 p_6 +12 p_2 p_4^2 + 16 p_2^3 p_4 \notag \\
& \quad  + 4 p_2^5 + 8 p_4 p_6  + 4 \lambda_4 (p_6 +2 p_2 p_4 + p_2^3)
+ 4 \lambda_6 (p_4 + p_2^2) + 4 \lambda_8 p_2 + \lambda_{10} = 0, \notag\\
&\J_{12} \equiv {-} 2 q_5 q_7 - 2 p_4 q_3 q_5 - (p_6 + p_2 p_4) q_3^2
+ 16 p_2 p_4 p_6 +12 p_2^2 p_4^2 + 4 p_6^2  + 4 p_2^3 p_6  \label{JacRelsC27}\\
&\quad + 4 p_2^4 p_4 + 4 p_4^3 + 4 \lambda_4 (p_2 p_6 + p_4^2 + p_2^2 p_4)
+ 4 \lambda_6 (p_6 + p_2 p_4) \notag \\
&\quad + 4 \lambda_8 p_4 + 4 \lambda_{12} = 0, \notag\\
&\J_{14} \equiv {-} q_7^2 - 2 p_6 q_3 q_5 - p_2 p_6 q_3^2 + 8 p_2 p_6^2 + 4 p_4^2 p_6
+ 12 p_2^2 p_4 p_6 + 4 p_2^4 p_6 \notag\\
&\quad + 4 \lambda_4 p_6 (p_4 + p_2^2)
+ 4 \lambda_6 p_2 p_6 + 4 \lambda_8 p_6 + 4 \lambda_{14} = 0, \notag
\end{align}
written in terms of the basis $\wp$-functions in $\mathfrak{A}(\mathcal{C})$: 
\begin{gather}\label{BasisWPC27}
p_2 = \wp_{1,1},\ \ p_4 = \wp_{1,3},\ \  p_6 = \wp_{1,5},\ \ 
q_3 = \wp_{1,1,1},\ \  q_5 = \wp_{1,1,3},\ \  q_7 = \wp_{1,1,5}.
\end{gather}

The same equations of $\Jac(\mathcal{C}) \backslash \Sigma$ are derived from
 the fundamental cubic relations given in the form \eqref{JacEqCHE} with 
 \begin{gather*}
 \mathsf{H}  =  \mathsf{P}  -  \mathsf{L} ,\qquad
 \mathsf{P} = \begin{pmatrix} 
 0 & 0 & \wp_{5,5} & \wp_{3,5} & \wp_{1,5} \\
 0 & -2 \wp_{5,5} & -\wp_{3,5} & \wp_{3,3} - 2 \wp_{1,5} & \wp_{1,3} \\
 \wp_{5,5} & -\wp_{3,5} & 2 \wp_{1,5} - 2\wp_{3,3} & -\wp_{1,3} & \wp_{1,1} \\
 \wp_{3,5} & \wp_{3,3} - 2 \wp_{1,5} & -\wp_{1,3} & -2\wp_{1,1} & 0\\
 \wp_{1,5} & \wp_{1,3} & \wp_{1,1} & 0 & 0
 \end{pmatrix},\\
 \mathsf{L} = \begin{pmatrix}
  2\lambda_{14} & \lambda_{12} & 0 & 0 & 0\\
  \lambda_{12} & 2\lambda_{10} & \lambda_8 & 0 & 0 \\
  0 & \lambda_8 & 2\lambda_6 & \lambda_4 \\
  0 & 0 & \lambda_4 & 0 & 1 \\
  0 & 0 & 0 &1 & 0
   \end{pmatrix}, \quad
  \mathsf{T} = \begin{pmatrix}  
  1 & 0 & 0\\
  0 & 1 & 0\\
  0 & 0 & 1\\
  \wp_{1,5} & \wp_{1,3} & \wp_{1,1} \\
  \wp_{1,1} \wp_{1,5} &   \wp_{1,1} \wp_{1,3} + \wp_{1,5} &   \wp_{1,1}^2 + \wp_{1,3}
 \end{pmatrix},
 \end{gather*}
 after elimination of $\wp_{3,3}$, $\wp_{3,5}$, $\wp_{5,5}$.
\end{27Curve}

\subsection{Identities for $\wp$-functions on $\Jac(\mathcal{C}) \backslash \Sigma$}
Having in mind Theorem \ref{T:BasisFunctHE} and Conjecture~\ref{C:BasisFunct}, 
we aim at finding rational (or polynomial) expressions for all $\Phi \in \mathfrak{A}(\mathcal{C})$ through basis $\wp$-functions.
We will use the residue theorem technique and Theorem~\ref{JRels}.

The main polynomial functions $\mathcal{R}_{2g-1 + \mathfrak{w}_i}$, 
obtained from \eqref{rExpr}, and reduced to basis monomials in $\mathfrak{P}(\mathcal{C})$,
produce identities for $\wp$-functions on $\Jac(\mathcal{C})\backslash \Sigma$. Actually,
\begin{multline}\label{RhighRels}
\mathcal{R}_{2g-1 + \mathfrak{w}_i}(x,y;u) - \mathcal{M}(x,y) \mathcal{R}_{2g}(x,y;u)
- \mathcal{N}(x,y) \mathcal{R}_{2g+1}(x,y;u) \\
= \sum_{j=1}^g  \upsilon_{\mathfrak{w}_j} (x,y) \F_{\mathfrak{w}_i+\mathfrak{w}_j} (u),\quad i \geqslant n.
\end{multline}
Then the identities have the form
\begin{equation}
\F_{\mathfrak{w}_i+\mathfrak{w}_j} (u) = 0,\quad i\geqslant n,\ \ j=1,\dots, g.
\end{equation}
Explicit expressions $\F_{\mathfrak{w}_i+\mathfrak{w}_j}$ 
are obtained after finding unknown coefficients of $\mathcal{M}$, and $\mathcal{N}$ of weights
$\wgt \mathcal{M} = \mathfrak{w}_i-1$, $\wgt \mathcal{N} = \mathfrak{w}_i-2$.

In the hyperelliptic case, $\F_w$ have a certain parity.
In the trigonal case, $\F_w$ have no certain parity, and 
split into even and odd parts:
\begin{gather}
\F_w^{\text{o}} = \tfrac{1}{2} \big(\F_w(u) - \F_w(-u) \big),\qquad 
\F_w^{\text{e}} = \tfrac{1}{2} \big(\F_w(u) + \F_w(-u) \big).
\end{gather}

\begin{34Curve}
Expressions  for $\wp_{2,2}$, and $\wp_{2,5}$ in terms of the basis $\wp$-functions \eqref{BasisWPC34}
are given by \eqref{wp22wp25C34}. 
Further, we use these two $\wp$-functions as an extension of the basis $\wp$-functions, 
in order to obtain simpler expressions. Except $\mathcal{R}_{6}$, $\mathcal{R}_{7}$,
which give a solution of the Jacobi inversion problem, and produce the identities \eqref{GenCubRelsC34},
we have the main polynomial function
\begin{multline}
\mathcal{R}_{10}(x,y;u) 
- \Big(\tfrac{5}{2} (\wp_{2,2}(u) - \wp_{1,1,2}(u)\big) + \tfrac{2}{3} \lambda_2^2  \Big) \mathcal{R}_{6}(x,y;u) \\
- \Big(\tfrac{5}{2} x + \tfrac{5}{4} (\wp_{1,2}(u) - \wp_{1,1,1}(u)\big) \Big) \mathcal{R}_{7}(x,y;u) 
= y \F_{6} (u) + x \F_{7} (u) + \F_{10} (u).
\end{multline}
Then we find
\begin{gather}
\begin{split}
\F_{5+\mathfrak{w}_j}^{\text{o}} &=
 - \tfrac{5}{4} \wp_{1,2} \wp_{1,1,\mathfrak{w}_j} 
-  \tfrac{5}{2} \wp_{1,\mathfrak{w}_j} \wp_{1,1,2}
 - \tfrac{5}{4} \wp_{2,\mathfrak{w}_j} \wp_{1,1,1} 
 -  \tfrac{5}{12} \lambda_2 \wp_{1,2,\mathfrak{w}_j}  \\
&\quad + \tfrac{5}{12}  \wp_{1,1,1,2,\mathfrak{w}_j} -  \tfrac{5}{2} \wp_{1,1,5} \, \delta_{j,2},  \\
\F_{5+\mathfrak{w}_j}^{\text{e}} &= -  \wp_{5,\mathfrak{w}_j}
+ \tfrac{5}{4} \wp_{1,2} \wp_{2,\mathfrak{w}_j} 
+ \big(\tfrac{5}{2} \wp_{2,2} + \tfrac{2}{3} \lambda_2^2 \big) \wp_{1,\mathfrak{w}_j}
+ \tfrac{5}{4} \wp_{1,1,1} \wp_{1,1,\mathfrak{w}_j} \\
&\quad  - \tfrac{5}{8} \wp_{1,2,2,\mathfrak{w}_j}  - \tfrac{1}{24} \wp_{1,1,1,1,1,\mathfrak{w}_j} 
+ \lambda_6 \delta_{j,1} + \big(\tfrac{5}{2} \wp_{2,5} + \tfrac{2}{3} \lambda_2 \lambda_5 \big)\delta_{j,2}.
\end{split}
\end{gather}
From $\F_w$ and $\mathcal{G}_w$ we obtain
\begin{gather}\label{WPEsIdent}
\begin{split}
&\wp_{1,1,1,1} =  \wp_{1,1} (6\wp_{1,1} + 4 \lambda_2) - 3 \wp_{2,2}, \\
&\wp_{1,1,1,2} =  \wp_{1,2} (6\wp_{1,1}+\lambda_2 ) + \lambda_5,  \\
&\wp_{1,1,1,5} =  \wp_{1,5} (6\wp_{1,1}+\lambda_2 ) + \lambda_8 
+ \tfrac{3}{4} \big( {-}\wp_{1,1,2}^2 + 4 \wp_{1,1} ( \wp_{1,2}^2 + \lambda_6 )
+ \wp_{2,2}^2  \big),\\
&\wp_{5,5} = \wp_{1,1,2}^2 \big(\tfrac{1}{2} (\wp_{1,1} + \lambda_2) - \tfrac{3}{8} \wp_{1,1}^{-1} \wp_{2,2} \big)
+ \wp_{2,5} \big( 2\wp_{1,2} + \tfrac{1}{2}\wp_{1,1}^{-1} (\lambda_2 \wp_{1,2} + \lambda_5) \big) \\
&\qquad - \tfrac{1}{8} \wp_{1,1}^{-1} \wp_{2,2}^3 + 
+ \tfrac{1}{2} \wp_{2,2} \big( \wp_{1,2}^2 - \lambda_6 + \wp_{1,1}^{-1} (\lambda_2 (\wp_{1,2}^2 + \wp_{1,5})
+ \lambda_5 \wp_{1,2} + \lambda_8)  \big)\\
&\qquad - 2 (\wp_{1,1}+\lambda_2) \wp_{1,1} \wp_{1,2}^2 
- 2 \lambda_5 \wp_{1,1} \wp_{1,2} - \lambda_2 \lambda_5 \wp_{1,2} 
- \tfrac{2}{3} \lambda_2 \lambda_8 - \tfrac{1}{2} \lambda_5^2 \\
&\qquad + \wp_{1,1}^{-1} \big(  \tfrac{1}{2} \wp_{1,2}^4 + \wp_{1,5}^2 + 2 \wp_{1,2}^2 \wp_{1,5} 
+ \lambda_6 (\tfrac{1}{2} \wp_{1,2}^2 + \wp_{1,5}) + \tfrac{1}{2} \lambda_9 \wp_{1,2} + \lambda_{12}  \big).
\end{split}
\end{gather}
With these expressions and  \eqref{wp22wp25C34} any $\wp$-function can be represented as
a rational function of the basis $\wp$-functions \eqref{BasisWPC34}.
Indeed, every derivative of expressions in \eqref{WPEsIdent} is a rational function,
expressible through the basis $\wp$-functions by means of \eqref{wp22wp25C34} and \eqref{WPEsIdent}.
\end{34Curve}

\section{Algebraic models of Kummer variety}
In \cite[\S\,4]{belHKF1996} we find an explicit realization of  $\Kum(\mathcal{C})$ of a hyperelliptic curve $\mathcal{C}$.
Let $\mathsf{H}$ be as defined in \eqref{KleinFMatr}, \eqref{PHmatr},
and  $\mathsf{K}(\mathsf{H})$ be  a $g \times g$ matrix with entries 
\begin{equation}\label{KumMatr}
\mathsf{k}_{i,j} = \det \big( \mathsf{H}[{}^{i,g+1,g+2}_{j,g+1,g+2}] \big),
\end{equation}
where $\mathsf{H}[{}^{i,k,l}_{j,m,n}]$ denotes a minor  of order $3$ of  matrix $\mathsf{H}$ composed 
as follows
$$\mathsf{H}[{}^{i,k,l}_{j,m,n}] = 
\begin{vmatrix} 
\mathsf{h}_{i,j} & \mathsf{h}_{i,m}  & \mathsf{h}_{i,n} \\
\mathsf{h}_{k,j}  & \mathsf{h}_{k,m}  & \mathsf{h}_{k,n} \\
\mathsf{h}_{l,j}  & \mathsf{h}_{l,m}  & \mathsf{h}_{l,n}
\end{vmatrix}.$$ 
The fundamental cubic relations associated with $\mathcal{C}$, see \cite[Eq.\,(3.15)]{belHKF1996}, 
acquire the form
\begin{equation}\label{KMatr}
\tfrac{1}{2} \wp_{1,1,\mathfrak{w}_i} \wp_{1,1,\mathfrak{w}_j} = \mathsf{k}_{i,j}.
\end{equation}
Evidently, $\rank \mathsf{K}(\mathsf{H})$ does not exceed $1$.

\begin{Theorem}[Hyperelliptic case]
Minors of order $2$ of the matrix $\mathsf{K}(\mathsf{H})$ 
define $\Kum(\mathcal{C})$.
\end{Theorem}

\begin{proof}
There are $\binom{g}{2} = \tfrac{1}{2} g (g-1)$ independent  minors of order $2$ of $\mathsf{K}(\mathsf{H})$
on $\Jac(\mathcal{C}) \backslash \Sigma = \mathcal{A}(\mathfrak{C}_g)$, since $\rank \mathsf{K}(\mathsf{H}) = 1$.
The minors are expressed in terms of $2$-index $\wp$-functions, which are even.
All $\tfrac{1}{2} g (g+1)$ functions $\wp_{\mathfrak{w},\mathfrak{j}}$, $i,j=1$, \ldots,$g$, are involved,
 which follows from the construction of matrix $\mathsf{P}$. 
Thus, minors of order $2$ of $\mathsf{K}(\mathsf{H})$ define a 
$g$-dimensional subspace in $\mathcal{A}(\mathfrak{C}_g) / \pm$.
\end{proof}
\begin{Remark}
The matrix form of the fundamental cubic relations gives
$\mathsf{K}(\mathsf{H}) = -\mathsf{T}^t \mathsf{H} \mathsf{T}$.
In the hyperelliptic case, a proof can be found in \cite[Sect.\,3]{belHKF1996}.
Similar matrix construction in the case of trigonal curves is suggested in \cite[Sect.\,2]{bel2000}.
\end{Remark}
\begin{27Curve}
Equations which define $\Kum(\mathcal{C})$ can be constructed as follows
\begin{gather}
\mathsf{k}_{1,1} \mathsf{k}_{3,3} - \mathsf{k}_{1,3}^2 = 0,\qquad
\mathsf{k}_{1,1} \mathsf{k}_{5,5} - \mathsf{k}_{1,5}^2 = 0,\qquad
\mathsf{k}_{3,3} \mathsf{k}_{5,5} - \mathsf{k}_{3,5}^2 = 0.
\end{gather}
\end{27Curve}

\begin{34Curve}
The expression for $\wp_{1,1,1,5}$ from \eqref{WPEsIdent} allows to split  
$\mathcal{G}_{10}$, $\mathcal{G}_{11}$, $\mathcal{G}_{14}$ in \eqref{GenCubRelsC34}, 
and obtain fundamental cubic relations associated with a $(3,4)$-curve,
namely
\begin{align}
&\mathcal{G}_6 = 0,\qquad\qquad \mathcal{G}_7=0,\notag\\
&\widetilde{\mathcal{G}}_{8} \equiv {-} \wp_{1,1,2}^2 + 4 \wp_{1,1} \wp_{1,2}^2 + \wp_{2,2}^2 + 4 \lambda_6 \wp_{1,1}
  + 8 \wp_{1,1} \wp_{1,5} \notag \\
&\phantom{\mathcal{G}_6 =\ }  - \tfrac{4}{3} \big(\wp_{1,1,1,5} - \lambda_2 \wp_{1,5} - \lambda_8 \big) = 0,\notag\\
&\widetilde{\mathcal{G}}_{10} \equiv {-}2\wp_{1,1,1} \wp_{1,1,5} + 2 \wp_{1,2} \wp_{2,5} - 4 \wp_{1,5} \wp_{2,2} \notag \\
&\phantom{\mathcal{G}_6 =\ }  
 + \tfrac{4}{3} \wp_{1,1} \big(\wp_{1,1,1,5} + 2 \lambda_2 \wp_{1,5} 
+ 2 \lambda_8 \big) =0,\label{FundCubRelsC34} \\
&\widetilde{\mathcal{G}}_{11} \equiv  {-}2\wp_{1,1,2} \wp_{1,1,5} + 2 \wp_{2,2} \wp_{2,5} + 4 \lambda_9 \wp_{1,1}\notag \\
&\phantom{\mathcal{G}_6 =\ }  
 + \tfrac{4}{3} \wp_{1,2} \big(\wp_{1,1,1,5} - \lambda_2 \wp_{1,5} - \lambda_8 \big)=0,\notag \\
&\widetilde{\mathcal{G}}_{14} \equiv  {-} \wp_{1,1,5}^2 - 4 \wp_{1,1} \wp_{1,5}^2  
+ \wp_{2,5}^2 + 4 \lambda_{12} \wp_{1,1}\notag \\
&\phantom{\mathcal{G}_6 =\ }  
 + \tfrac{4}{3} \wp_{1,5} \big(\wp_{1,1,1,5} - \lambda_2 \wp_{1,5} - \lambda_8 \big)=0. \notag
\end{align}
Equations \eqref{FundCubRelsC34} define the matrix $\mathsf{K} = (\mathsf{k}_{i,j})$ by \eqref{KMatr}.
All entries of $\mathsf{K}$ are expressed through $6$ even functions: 
$\wp_{1,1}$, $\wp_{1,2}$, $\wp_{1,5}$, $\wp_{2,2}$, $\wp_{2,5}$, $\wp_{1,1,1,5}$.
Three independent equations which define $\Kum(\mathcal{C})$ can be constructed as follows
\begin{gather}
\mathsf{k}_{1,1} \mathsf{k}_{2,2} - \mathsf{k}_{1,2}^2 = 0,\qquad
\mathsf{k}_{1,1} \mathsf{k}_{5,5} - \mathsf{k}_{1,5}^2 = 0,\qquad
\mathsf{k}_{2,2} \mathsf{k}_{5,5} - \mathsf{k}_{2,5}^2 = 0.
\end{gather}
\end{34Curve}

\section{Addition laws on Jacobian varities}
Every abelian function field $\mathfrak{A}(\mathcal{C})$ possesses the addition law,
which means that every $\wp$-function at $u+\tilde{u}$ can be expressed 
in terms of basis $\wp$-functions at $u$ and $\tilde{u}$.

Addition theorems in $\mathfrak{A}(\mathcal{C})$ associated with elliptic curves are known in several forms:
\begin{itemize}
\item the addition formula for $\sigma$-function
\begin{equation}\label{AddLawSigma}
\frac{\sigma(u+\tilde{u}) \sigma(u-\tilde{u})}{\sigma(u)^2 \sigma(\tilde{u})^2}
= \wp(\tilde{u}) - \wp(u);
\end{equation}
\item the addition formula for $\zeta$-function, derivable from \eqref{AddLawSigma}, 
\begin{equation}\label{FSg1}
\big(\zeta(u) + \zeta(\tilde{u})  + \zeta(\hat{u})\big)^2 = 
\wp(u) + \wp(\tilde{u})  + \wp(\hat{u}),\qquad  
u + \tilde{u} + \hat{u} = 0;
\end{equation}
\item the addition formula for $\wp$-function, which can be obtained by differentiating  \eqref{FSg1},
\begin{equation}\label{AddP2G1}
\wp(u+\tilde{u}) = - \wp(u) - \wp(\tilde{u}) + \frac{1}{4} 
\bigg(\frac{\wp'(u) - \wp'(\tilde{u}) }{\wp(u) - \wp(\tilde{u}) }\bigg)^2.
\end{equation}
\end{itemize}

Below, we consider what was done recently 
in the direction of extending these results to curves of higher genera.

\subsection{Groupoid structure of Jacobian varieties}
In \cite{BL2005,BL2005ru} the addition law is derived from the groupoid structure of $\Jac(\mathcal{C})$.

\begin{Definition}
A space $\textsf{X}$  together with an anchor mapping $\textsf{p}_{\textsf{X}}: \textsf{X} \to \textsf{Y}$
is called a groupoid over $\textsf{Y}$, if the two structure mappings over $\textsf{Y}$:
\begin{equation}
\add: \textsf{X} \times_{\textsf{Y}} \textsf{X} \to \textsf{X}, \qquad
\inv \textsf{X} \to \textsf{X}, 
\end{equation}
are defined and  satisfy the axioms
\begin{enumerate}
\item[1.] $\add(\add(\textsf{x}_1,\textsf{x}_2),\textsf{x}_3) 
= \add(\textsf{x}_1,\add(\textsf{x}_2, \textsf{x}_3))$, provided 
$\textsf{p}_{\textsf{X}}(\textsf{x}_1)=\textsf{p}_{\textsf{X}}(\textsf{x}_2) = \textsf{p}_{\textsf{X}}(\textsf{x}_3)$,
\item[2.] $\add(\add(\textsf{x}_1,\textsf{x}_2),\inv(\textsf{x}_2)) = \textsf{x}_1$, provided 
$\textsf{p}_{\textsf{X}}(\textsf{x}_1) = \textsf{p}_{\textsf{X}}(\textsf{x}_2)$.
\end{enumerate}
\end{Definition}
A groupoid structure is called \emph{commutative}, if $\add(\textsf{x}_1,\textsf{x}_2)=\add(\textsf{x}_2,\textsf{x}_1)$,
and \emph{algebraic} if $\textsf{p}_{\textsf{X}}$
as well as  $\add$, and $\inv$ are algebraic.

Let $\textsf{Y}$ be $\Lambda$, the space of parameters of an algebraic curve $\mathcal{C}$,
and $\textsf{X}$ be $\mathcal{E}(\Lambda,\pi,\mathcal{C}^g)$, the bundle of the $g$-th symmetric power of $\mathcal{C}$.
Let $\textsf{x}$ be a degree $g$ non-special divisor $D \in \mathfrak{C}_g$, and $\textsf{p}_{\textsf{X}}(D) = \lambda$.

\begin{Theorem} \label{T:AddLaw} \cite[Lemmas 2.2, 2.3]{BL2005}
The polynomial function $\mathcal{R}_{2g}$ of order $2g$ on $\mathcal{C}$ defines the inverse mapping $\inv$,
and  the polynomial function $\mathcal{R}_{3g}$ of order $3g$  defines the addition mapping $\add$.
\end{Theorem}
\begin{Remark}
Theorem~\ref{T:AddLaw} reflects the fact that for any $w\geqslant 2g$
a polynomial function $\mathcal{R}_w$ is uniquely defined by a degree $w-g$ divisor $D$
and produces an inverse divisor $D^\ast \in \mathfrak{C}_g$ such that $(\mathcal{R}_w)_0 = D+D^\ast$,
according to Theorem~\ref{T:PolyFunct} and Corollary~\ref{C:AddP}.

The function $\mathcal{R}_{2g}$ is defined over $\Jac(\mathcal{C})\backslash \Sigma$,
and produces an inverse from $\mathfrak{C}_g = \mathcal{A}^{-1} (\Jac(\mathcal{C})\backslash \Sigma)$,
that is $\mathcal{A}(D) = u \in \Jac(\mathcal{C})\backslash \Sigma$,
as well as $\mathcal{A}(D^\ast) = -u$.

The function $\mathcal{R}_{3g}$ is defined over $\big(\Jac(\mathcal{C})\backslash \Sigma \big)^2$,
that is $\mathcal{A}(D) = u + \tilde{u}$, $u, \tilde{u} \in \Jac(\mathcal{C})\backslash \Sigma$, and produces
$\hat{u} = \mathcal{A}(D^\ast)$, such that $\hat{u} = -( u+\tilde{u})$. 
\end{Remark}

In \cite{BL2005,BL2005ru}, this approach
is illustrated by obtaining  addition laws explicitly in the hyperelliptic case,
which are given in a concise form for an arbitrary genus. 
The addition law contains expressions for all basis $\wp$-function at $u+\tilde{u}$ 
in terms of basis $\wp$-functions at $u$ and $\tilde{u}$.
Below, we show the process of obtaining such expressions in genus $3$.

\begin{27Curve}
Let $D$, $\tilde{D} \in \mathfrak{C}_g$, and 
$u = \mathcal{A}(D)$,  $\tilde{u} = \mathcal{A}(\tilde{D})$.
Each divisor is defined by the two functions $\mathcal{R}_{6}$, 
$\mathcal{R}_{7}$ of the form \eqref{JIPC27}
computed at $u$, and $\tilde{u}$, or by the two collections of basis $\wp$-functions:
\begin{gather}\label{BasisFunct2D}
\begin{split}
&D:\qquad\wp_{1,1}(u), \wp_{1,3}(u), \wp_{1,5}(u),
\wp_{1,1,1}(u), \wp_{1,1,3}(u), \wp_{1,1,5}(u); \\
&\tilde{D}:\qquad \wp_{1,1}(\tilde{u}), \wp_{1,3}(\tilde{u}), \wp_{1,5}(\tilde{u}),
\wp_{1,1,1}(\tilde{u}), \wp_{1,1,3}(\tilde{u}), \wp_{1,1,5}(\tilde{u}).
\end{split}
\end{gather}
Let a polynomial function of weight $3g=9$ have the form
\begin{equation}\label{R9DefC27}
\mathcal{R}_{9}(x,y;\gamma) = yx + \gamma_1 x^4 + \gamma_2 y + \gamma_3 x^3 
+ \gamma_5 x^2 + \gamma_7 x + \gamma_9,
\end{equation}
and $D+\tilde{D} \subset (\mathcal{R}_{9})_0$. According to Theorem~\ref{T:PolyFunct},
this divisor defines $\mathcal{R}_{9}$ uniquely, which means $\gamma$ are expressible 
in terms of the basis $\wp$-functions \eqref{BasisFunct2D}.
First, we reduce $\mathcal{R}_{9}$ with the help of $\mathcal{R}_{6}$, and $\mathcal{R}_{7}$, namely
\begin{multline}
\mathcal{R}_{9}(x,y;\gamma)  - \big(\gamma_1 x 
+ \gamma_3 + \gamma_1 \wp_{1,1}(u)  - \tfrac{1}{2} \wp_{1,1,1}(u) \big) \mathcal{R}_{6}(x;u) 
- \tfrac{1}{2}\big(x + \gamma_2 \big) \mathcal{R}_{7}(x,y;u) \\
= x^2 \F_{5}(u;\gamma) + x \F_{7}(u;\gamma) + \F_{9}(u;\gamma).
\end{multline}
This implies
\begin{equation}\label{SEqs}
\F_{w}(u;\gamma)=0,\quad \F_{w}(\tilde{u};\gamma)=0,\quad w=5,7,9,
\end{equation}
or in the matrix form
\begin{subequations}\label{AddLawEqs}
\begin{equation}
\begin{pmatrix} 
1_3 & \mathsf{A}(u) \\
1_3 & \mathsf{A}(\tilde{u}) 
\end{pmatrix} 
\begin{pmatrix} \breve{\gamma} \\ \bar{\gamma}
\end{pmatrix} \\
+ \begin{pmatrix} 
\mathsf{b}(u) \\
\mathsf{b}(\tilde{u}) 
\end{pmatrix} = 0,
\end{equation}
where $1_3$ denotes the identity matrix of order $3$, and
\begin{gather}
\begin{split}
&\mathsf{A}(u) = 
\begin{pmatrix} 
\wp_{1,5}(u) & - \tfrac{1}{2} \wp_{1,1,5}(u) & \wp_{1,1}(u) \wp_{1,5}(u) \\
\wp_{1,3}(u) & - \tfrac{1}{2} \wp_{1,1,3}(u) & \wp_{1,1}(u) \wp_{1,3}(u) + \wp_{1,5}(u) \\
\wp_{1,1}(u) & - \tfrac{1}{2} \wp_{1,1,1}(u) & \wp_{1,1}(u)^2 + \wp_{1,3}(u)
\end{pmatrix}, \\
&\mathsf{b}(u) = \begin{pmatrix} 
- \tfrac{1}{2} \wp_{1,1,1}(u) \wp_{1,5}(u)\\
- \tfrac{1}{2} \wp_{1,1,1}(u) \wp_{1,3}(u) - \tfrac{1}{2} \wp_{1,1,5}(u)\\
- \tfrac{1}{2} \wp_{1,1,1}(u) \wp_{1,1}(u) - \tfrac{1}{2} \wp_{1,1,3}(u)
 \end{pmatrix},\quad
 \breve{\gamma} = \begin{pmatrix} \gamma_9 \\ \gamma_7 \\ \gamma_5  \end{pmatrix},\quad
 \bar{\gamma} = \begin{pmatrix} \gamma_3 \\ \gamma_2 \\ \gamma_1 \end{pmatrix}.
 \end{split}
\end{gather}
\end{subequations}
Note, that $\mathsf{A}(u) = (\Upsilon_1,\Upsilon_2,\Upsilon_3)$, cf. \eqref{WPCols},
and $\mathsf{b}(u) = {-}\tfrac{1}{2} \wp_{1,1,1} \Upsilon_1 + \Upsilon_2^\circ$,
where $\Upsilon_2^\circ = (0$, ${-}\tfrac{1}{2} \wp_{1,1,5}$, ${-}\tfrac{1}{2} \wp_{1,1,3})^t$.
Then,
\begin{gather}\label{GammaC27}
\begin{split}
& \bar{\gamma} = - \big(\mathsf{A}(u)-\mathsf{A}(\tilde{u})\big)^{-1} 
\big(\mathsf{b}(u) - \mathsf{b}(\tilde{u})\big),\\
& \breve{\gamma} = - \mathsf{b}(u) + \mathsf{A}(u)
\big(\mathsf{A}(u)-\mathsf{A}(\tilde{u})\big)^{-1} 
\big(\mathsf{b}(u) - \mathsf{b}(\tilde{u})\big).
\end{split}
\end{gather}

On the other hand, the polynomial function $\mathcal{R}_{9}$ 
defined by \eqref{R9DefC27} and \eqref{GammaC27} produce a divisor $\hat{D}\in\mathfrak{C}_3$
such that $(\mathcal{R}_{9})_0 = D+\tilde{D} + \hat{D}$.
Let $\hat{u} = \mathcal{A}(\hat{D})$, and we also have
\begin{equation}\label{u3R9}
\breve{\gamma} +  \mathsf{A}(\hat{u}) \bar{\gamma} +  \mathsf{b}(\hat{u}) = 0.
\end{equation}
Let $D^\ast = {-}D$, $\tilde{D}^\ast = {-}\tilde{D}$, $\hat{D}^\ast = {-} \hat{D}$, which implies
$(\mathcal{R}_{6}(u))_0 = D+D^\ast$, $(\mathcal{R}_{6}(\tilde{u}))_0 = \tilde{D}+ \tilde{D}^\ast$,
$(\mathcal{R}_{6}(\hat{u}))_0 = \hat{D} + \hat{D}^\ast$. Now, we construct $\mathcal{R}_{9}^-$ by the map $u \mapsto -u$,
and so $(\mathcal{R}_{9}^-)_0 = D^\ast +\tilde{D}^\ast + \hat{D}^\ast$.
As seen from \eqref{AddLawEqs}, $\gamma$ with odd indices are odd functions in $u$,
and $\gamma$ with even index is an even function. Thus,
$$ \mathcal{R}_{9}^{-}(x,y;\gamma) = yx - \gamma_1 x^4 + \gamma_2 y 
- \gamma_3 x^3 - \gamma_5 x^2 - \gamma_7 x - \gamma_9, $$
Then the equality
\begin{equation}
\mathcal{R}_{9}(x,y;\gamma)\mathcal{R}_{9}^{-}(x,y;\gamma) 
- \mathcal{R}_{6}(x;u) \mathcal{R}_{6}(x;\tilde{u}) \mathcal{R}_{6}(x;\hat{u})
= - (x+\gamma_2)^2 f (x, y;\lambda)
\end{equation}
allows to determine the required divisor $\hat{D}= - (D+\tilde{D})$. In fact, 
$\mathcal{R}_{9}(x,y;\gamma)\mathcal{R}_{9}^{-}(x,y;\gamma) + (x+\gamma_2)^2 f (x, y;\lambda)$
is a polynomial in $x$ of degree $9$, and
coefficients of $x^8$, $x^7$, $x^6$ produce
\begin{gather}
\begin{split}
&\wp_{1,1}(\hat{u})  = - \wp_{1,1}(u) - \wp_{1,1}(\tilde{u}) - 2 \gamma_2 + \gamma_1^2,\\
&\wp_{1,3}(\hat{u})  = - \wp_{1,3}(u) - \wp_{1,3}(\tilde{u}) + \wp_{1,1}(u) \wp_{1,1}(\tilde{u}) 
+ \big( \wp_{1,1}(u) + \wp_{1,1}(\tilde{u}) \big) \wp_{1,1}(\hat{u}) \\
&\phantom{\wp_{1,3}(u)  =} 
 + 2 \gamma_1 \gamma_3 - \gamma_2^2 - \lambda_4, \\
&\wp_{1,5}(\hat{u}) = - \wp_{1,5}(u) - \wp_{1,5}(\tilde{u})
+ \wp_{1,3}(u) \wp_{1,1}(\tilde{u}) + \wp_{1,3}(\tilde{u}) \wp_{1,1}(u) \\
&\phantom{\wp_{1,3}(u)  =} 
+ \big( \wp_{1,1}(u) + \wp_{1,1}(\tilde{u})\big) \wp_{1,3}(\hat{u})  \\
&\phantom{\wp_{1,3}(u)  =} 
 + \big( \wp_{1,3}(u) + \wp_{1,3}(\tilde{u}) - \wp_{1,1}(u) \wp_{1,1}(\tilde{u}) \big) \wp_{1,1}(\hat{u})  \\
&\phantom{\wp_{1,3}(u)  =} 
   + \gamma_3^2 + 2 \gamma_1 \gamma_5  - 2 \gamma_2 \lambda_4  - \lambda_6.
\end{split}
\end{gather}
Finally, 
$\wp_{1,1,1}(\hat{u})$, $\wp_{1,1,3}(\hat{u})$, $\wp_{1,1,5}(\hat{u})$
are obtained from \eqref{u3R9}.
\end{27Curve}

\subsection{Polylinear relations }
\begin{Definition}
We introduce the primitive function
\begin{equation}\label{PsiDef}
\psi(P) = \exp \Big({-} \int_{\infty}^P \mathcal{B}(\tilde{P}) \rmd u(\tilde{P}) \Big),\quad P=(x,y)\in\mathcal{C}.
\end{equation}
\end{Definition}

Let $\psi(\xi) = \psi \big(x(\xi),y(\xi) \big)$ be the expansion of $\psi$ 
in the vicinity of infinity in the local parameter $\xi$, introduced by \eqref{param}. 
Then, see \cite[Eq.\,2.2]{BL2005ru},
\begin{equation}\label{PsiExp}
\psi(\xi) = \xi^g \exp \Big( \mathcal{T}(\xi,\lambda) \Big),
\end{equation}
where $\mathcal{T}$ is a holomorphic function of its arguments $\xi$ and $\lambda$, $\mathcal{T}(\xi,0)=0$.
This implies, that $\psi$ is an entire function, and has the weight $-g$.
\begin{Theorem}\label{T:PsiDerSigma} \cite[Theorem 2.14]{BL2005ru}
Let $\mathfrak{n} = -\wgt \sigma - g$, then
\begin{equation}
\partial_{u_1}^\mathfrak{n} \sigma (u) \big|_{u \to \mathcal{A}(x,y)} = c \psi(x,y),
\end{equation}
where $c$ is constant.
\end{Theorem}

Generalizing \cite[Theorem 3.11]{BL2005ru}, we have
\begin{Theorem} 
Let $u^{[k]} \in \Jac(\mathcal{C}) \backslash \Sigma$, $k=1$, \ldots, $\ell \geqslant 2$, 
and $\mathcal{R}_{\ell g}$ be a monic polynomial function of weight $\ell g$ from $\mathfrak{P}(\mathcal{C})$
with $(\mathcal{R}_{\ell g})_0 = \sum_{k=1}^\ell \mathcal{A}^{-1}(u^{[k]})$. 
Then
\begin{equation}\label{PolyLinRelGen}
\mathcal{R}_{\ell g} (x,y) = \prod_{k=1}^\ell 
\frac{\sigma(u^{[k]} - \mathcal{A}(x,y))}{\psi(x,y) \sigma(u^{[k]})},\qquad
\sum_{k=1}^\ell u^{[k]} = 0.
\end{equation}
\end{Theorem}
The equalities \eqref{PolyLinRelGen}  generate  \emph{polylinear relations},
which serve as identities for $\wp$-functions of $u^{[k]}$, $k=1$, \ldots, $\ell$.
Indeed, sending $(x,y)$ to infinity, and applying \eqref{param}, the  expansion is obtain
\begin{equation}\label{PolyLinRelExp}
\mathcal{R}_{\ell g} (x,y) = \xi^{-\ell g} \bigg(1 + \sum_{s \geqslant 1} \xi^s 
T_s^{[\ell]} \big(u^{[1]},u^{[2]},\dots,u^{[\ell]}\big)  \bigg).
\end{equation}
Thus, with a fixed $\ell$ a sequence of
identities for $\wp$-functions at $u^{[1]}$, $u^{[2]}$, \ldots, $u^{[\ell]}$  is generated.

If $\ell=2$, the equality  \eqref{PolyLinRelGen} acquires the form
\begin{equation}\label{BiLinRelGen}
\mathcal{R}_{2g} (x,y) =
\frac{\sigma \big(u - \mathcal{A}(x,y)\big) \sigma \big(\tilde{u} - 
\mathcal{A}(x,y)\big)}{\psi(x,y)^2 \sigma(u)\sigma(\tilde{u})},\qquad
u+\tilde{u} = 0,
\end{equation}
where $\mathcal{R}_{2g}$ is the polynomial function of weight $2g$, 
which generates the inversion mapping on $\Jac(\mathcal{C})\backslash \Sigma$, 
and arises in a solution of the Jacobi inversion problem.
The equality  \eqref{BiLinRelGen},
after expanding in $\xi$ as $(x,y)\to \infty$, 
gives rise to a \emph{hierarchy of bilinear relations}, 
which are identities for $\wp$-functions of $u$, since $\tilde{u} = -u$.

If $\ell=3$, the equality \eqref{PolyLinRelGen} acquires the form
\begin{equation}\label{TriLinRelGen}
\mathcal{R}_{3g} (x,y) =
\frac{\sigma \big(u - \mathcal{A}(x,y)\big) \sigma \big(\tilde{u} - \mathcal{A}(x,y)\big)
\sigma \big(\hat{u} - \mathcal{A}(x,y)\big)}
{\psi(x,y)^3 \sigma(u) \sigma(\tilde{u}) \sigma(\hat{u})},\quad
u+\tilde{u} + \hat{u} = 0,
\end{equation}
where $\mathcal{R}_{3g}$ is the polynomial function of weight $3g$, 
which generates the addition mapping on $\Jac(\mathcal{C})\backslash \Sigma$.
The equality \eqref{TriLinRelGen} gives rise to \emph{trilinear relations}, 
which produce the addition law for $\wp$-functions, since $\hat{u} = {-}(\tilde{u} + u)$.

\begin{Remark}
The second kind integral $\mathcal{B}$ in the definition \eqref{PsiDef} 
requires to specify a regularization constant vector, see Remark~\ref{R:RegC}.
Thus,
\begin{equation*}
\mathcal{B}(\xi) = c(\lambda)+ \int_0^{\xi} \rmd r(\xi).
\end{equation*}
These constants  $c(\lambda)$ can be found by means of 
Theorem~\ref{T:PsiDerSigma}, if the series for $\sigma$-function is known,
or from comparison between identities for $\wp$-functions obtain from \eqref{BiLinRelGen},
and by the techniques described in section~\ref{s:AbelFunct}, for more detail see
\cite{BerLey2018}.
\end{Remark}

Further development in this direction, allows to represent polylinear relations in terms of operators.
Let $\mathcal{D} = (\mathcal{D}_{\mathfrak{w}_1}, \mathcal{D}_{\mathfrak{w}_2}, 
\dots, \mathcal{D}_{\mathfrak{w}_g})^t$ be symbols which stand for linear operators
\begin{equation}\label{DDefs}
\mathfrak{D} = \big(\mathfrak{D}_{\mathfrak{w}_1}, \mathfrak{D}_{\mathfrak{w}_2}, 
\dots, \mathfrak{D}_{\mathfrak{w}_g}\big)^t,\qquad
\mathfrak{D}_{\mathfrak{w}_i} = \sum_{k=1}^\ell \partial_{u^{[k]}_{\mathfrak{w}_i} },
\end{equation}
defined as follows
\begin{equation}
\mathcal{D}_{\mathfrak{w}_i} = \frac{
\mathfrak{D}_{\mathfrak{w}_i} \prod_{k=1}^\ell \sigma( u^{[k]} ) }
{ \prod_{k=1}^\ell \sigma( u^{[k]} )},\qquad \sum_{k=1}^\ell u^{[k]} = 0.
\end{equation}
Then, $T_s^{[\ell]}$ in \eqref{PolyLinRelExp} can be expressed in terms of $\mathcal{D}$.
We call $T_s^{[\ell]}$ \emph{polylinear operators}.

\begin{Theorem}\label{T:GenT} \cite[Lemma 3.13]{BL2005ru}
The generating function of polylinear operators  has the form 
\begin{equation}\label{PolyLinOpGen}
1 + \sum_{s \geqslant 1} \xi^s T_s^{[\ell]} 
= \exp \Big((\mathcal{A}(\xi), \mathcal{D}) - \ell \mathcal{T}(\xi, \lambda) \Big),\quad \ell\geqslant 2,
\end{equation}
where $\exp \mathcal{T}(\xi, \lambda) = \xi^{-g} \psi(\xi)$, cf.  \eqref{PsiExp}.
\end{Theorem}

\subsection{Bilinear relations and Baker---Hirota operators}
In \cite[\S\,13]{bakerMPF} the following operators are introduced:
\begin{equation}\label{BHOps}
\Delta_{\mathfrak{w}_i} = \frac{\partial}{\partial u_{\mathfrak{w}_i}} -  
\frac{\partial}{\partial \bar{u}_{\mathfrak{w}_i}}, \quad \bar{u}=u.
\end{equation}
Later, these operators were used by Hirota, and named after him. 
In \cite{AE2012} they are called Baker---Hirota operators, and used 
to produce identities for $\wp$-functions, with $\wp$-functions  defined as
\begin{gather}
\begin{split}
&\wp_{i,j}(u) = {-}\tfrac{1}{2} \sigma(u)^{-2} \Delta_{i} \Delta_{j} \sigma(u) \sigma(\bar{u}) \big|_{\bar{u}=u},\\
&Q_{i,j,k,l}(u) = {-}\tfrac{1}{2} \sigma(u)^{-2} \Delta_{i} \Delta_{j} \Delta_{k} \Delta_{l} \sigma(u) \sigma(\bar{u}) 
\big|_{\bar{u}=u}.
\end{split}
\end{gather}

\begin{Remark}
The linear operators $\mathfrak{D}$ used in constructing
bilinear operators $T_s^{[2]}$ are, in fact, the Baker---Hirota operators, cf. \eqref{DDefs}.
Therefore, bilinear relations produce dynamical equations.
\end{Remark}

Below, we illustrate how to use  bilinear operators in finding identities for $\wp$-funcitons.

\begin{34Curve}
On a $(3,4)$-curve we have $\psi(\xi) = \partial_{u_1}^2 \sigma(u)|_{u=u(\xi)}$,
and $c(\lambda) = (0,-\tfrac{1}{3}\lambda_2, -\frac{1}{6}\lambda_5)$, see \cite[Lemma 4.2, Theorem 4.1]{BerLey2018}.
Defining $\psi$  by \eqref{PsiDef}, we find   
\begin{equation}
\mathcal{T}(\xi,\lambda) = -\tfrac{1}{12} \lambda_5 \xi^5 - \tfrac{7}{60} \lambda_6 \xi^6 
- \tfrac{1}{90} \lambda_2 \lambda_5 \xi^7 
- \big(\tfrac{13}{168} \lambda_8 - \tfrac{3}{112} \lambda_2 \lambda_6\big) \xi^8 + O(\xi^9).
\end{equation}

Using \eqref{PolyLinOpGen} we generate $T_s^{[2]}$, and split each operator
 into odd $T_s^{[2,\text{o}]}$ and even $T_s^{[2,\text{e}]}$ parts. Note, that 
 $T_s^{[2,\text{o}]} = 0$ due to  $\tilde{u}=-u$.
Assuming, that $\mathcal{R}_{6}(x,y) = x^2 + \alpha_2 y + \alpha_3 x + \alpha_6$, we 
find the bilinear relations in terms of bilinear operators:
\begin{align*}
\alpha_2 &= T_2^{[2,\text{e}]} \equiv \tfrac{1}{2} \mathcal{D}_1^2, & \\
\alpha_3 &= T_3^{[2,\text{e}]} \equiv  \tfrac{1}{2} \mathcal{D}_1 \mathcal{D}_2,& \\
\tfrac{1}{3} \lambda_2 \alpha_2 &= T_4^{[2,\text{e}]} \equiv  
 \tfrac{1}{8} \mathcal{D}_2^2 + \tfrac{1}{4!}  \mathcal{D}_1^4,& \\
0 &= T_5^{[2,\text{e}]} \equiv  
\tfrac{1}{12} \big(\mathcal{D}_1^2 - \lambda_2 \big) \mathcal{D}_1 \mathcal{D}_2 
 + \tfrac{1}{6} \lambda_5,& \\
\alpha_6 &= T_6^{[2,\text{e}]} \equiv  \tfrac{1}{5} \mathcal{D}_1 \mathcal{D}_5 
+ \big(\tfrac{1}{16} \mathcal{D}_1^2  - \tfrac{1}{24} \lambda_2 \big) \mathcal{D}_2^2
+ \tfrac{1}{6!}  \mathcal{D}_1^6 - \tfrac{1}{45} \lambda_2^2 \mathcal{D}_1^2  
+ \tfrac{7}{30} \lambda_6,\quad \dots&
\end{align*}
By expanding the right hand side of \eqref{BiLinRelGen} in $\xi$,
we find the corresponding relations in terms of $\wp$-functions:
\begin{subequations}\label{BilinRelsC34}
\begin{align}
\alpha_2 &= -\wp_{1,1}(u),& \\
\alpha_3 &= -\wp_{1,2}(u),& \\
\tfrac{1}{3} \lambda_2 \alpha_2 &= \tfrac{1}{12} \big(6\wp_{1,1}(u)^2 - 3 \wp_{2,2}(u) - \wp_{1,1,1,1}(u)\big),& \label{T4Eq}\\
0 &= \tfrac{1}{6} \big(\lambda_5  + \lambda_2 \wp_{1,2}(u) + 6 \wp_{1,1}(u) \wp_{1,2}(u) - \wp_{1,1,1,2}(u)\big),& \label{T5Eq}\\
\alpha_6 & = - \wp_{1,5}(u) + \tfrac{7}{60} \big(2 \lambda_6 + \lambda_2 \wp_{2,2}(u) 
+ 4 \wp_{1,2}(u)^2 + 4 \wp_{1,5}(u) \label{T6Eq} \\
&\quad\ + 2 \wp_{1,1}(u) \wp_{2,2}(u) - \wp_{1,1,2,2}(u)\big),\quad \dots& \notag
\end{align}
\end{subequations}
From \eqref{T4Eq} and \eqref{T5Eq}
the first two identities in the list \eqref{WPEsIdent} are obtained in a simpler way. 
Then, \eqref{T6Eq} is simplified with the help of these two identities, and the cubic relation $\mathcal{G}_6$ from \eqref{GenCubRelsC34}. To proceed with producing identities, 
we need a solution of the Jacobi inversion problem, 
according to which $\alpha_6 = -\wp_{1,5}(u)$.

On the other hand, all relations \eqref{BilinRelsC34} come directly 
from expressions in terms of bilinear operators, and the latter are much easier to construct.
\end{34Curve}

\begin{Remark}
In the hierarchy of bilinear equations associated with a hyperelliptic curve, 
the identity $\wp_{1,1,1,1}(u)  -  6\wp_{1,1}(u)^2 - 4\wp_{1,3}(u) = 2 \lambda_4 $
represents the first integral of the Korteweg---de Vries (KdV) equation, 
which has the following Hirota's bilinear form,
cf. \cite[p.\,5]{NakST2010},
$$ -\tfrac{1}{2} \mathcal{D}_1^4 +  2 \mathcal{D}_1 \mathcal{D}_3 = 2\lambda_4. $$
Such an equation exists in the hierarchy in any genus  
greater than or equal to $2$.
\end{Remark}

\subsection{Trilinear relations and addition laws}
The structure of the hierarchy of trilinear operators associated with the family of
hyperelliptic curves is described in \cite[Theorem 3.15]{BL2005ru}.
Below, we give an example of constructing addition formulas on the simplest trigonal curve.
\begin{34Curve}
Assuming, that 
$\mathcal{R}_{9}(x,y)$ has the form \eqref{R9DefC27}, we 
generate trilinear relations in terms of operator symbols $\mathcal{D}$ 
by means of \eqref{PolyLinOpGen}, and
 split $T_s^{[3]}$ into odd $T_s^{[3,\text{e}]}$ and even $T_s^{[3,\text{o}]}$ parts:
\begin{gather*}
\begin{aligned}
\gamma_1^{\text{o}} &= T_1^{[3,\text{o}]} \equiv \mathcal{D}_1,&  
\gamma_1^{\text{e}} &= T_1^{[3,\text{e}]} \equiv 0,&\\
\gamma_2^{\text{o}} &= T_2^{[3,\text{o}]} \equiv \tfrac{1}{2} \mathcal{D}_2& 
\gamma_2^{\text{e}} &= T_2^{[3,\text{e}]} \equiv \tfrac{1}{2} \mathcal{D}_1^2, & \\
\gamma_3^{\text{o}} + \tfrac{2}{3} \lambda_2 \gamma_1^{\text{o}} &= T_3^{[3,\text{o}]} \equiv  
 \tfrac{1}{3!} \mathcal{D}_1^3,&
\gamma_3^{\text{e}} + \tfrac{2}{3} \lambda_2 \gamma_1^{\text{e}} &= T_3^{[3,\text{e}]} \equiv  
\tfrac{1}{2} \mathcal{D}_1 \mathcal{D}_2& \\
\tfrac{1}{3} \lambda_2 \gamma_2^{\text{o}} &= T_4^{[3,\text{o}]} \equiv   
\tfrac{1}{4} \mathcal{D}_1^2 \mathcal{D}_2 
- \tfrac{1}{12} \lambda_2 \mathcal{D}_2,&
\tfrac{1}{3} \lambda_2 \gamma_2^{\text{e}} &= T_4^{[3,\text{e}]} \equiv 
\tfrac{1}{8} \mathcal{D}_2^2 + \tfrac{1}{4!}  \mathcal{D}_1^4,& 
\end{aligned}\\
\begin{aligned}
\gamma_5^{\text{o}} + \tfrac{1}{9} \lambda_2^2 \gamma_1^{\text{o}} &= T_5^{[3,\text{o}]} \equiv  
 \tfrac{1}{5} \mathcal{D}_5 + \tfrac{1}{8} \mathcal{D}_1 \mathcal{D}_2^2 
  + \tfrac{1}{5!} \mathcal{D}_1^5
 - \tfrac{1}{45} \lambda_2^2 \mathcal{D}_1,& \phantom{mmmmtmmmmmmmmm}\\
\gamma_5^{\text{e}} + \tfrac{1}{9} \lambda_2^2 \gamma_1^{\text{e}} &= T_5^{[3,\text{e}]} \equiv  
 \tfrac{1}{12} \big(\mathcal{D}_1^2 - \lambda_2 \big) \mathcal{D}_1 \mathcal{D}_2
+ \tfrac{1}{4} \lambda_5,\quad \dots&
\end{aligned}
\end{gather*}
Let $\mathfrak{s}_{i,\dots}$ denote the sum ${-} (\wp_{i,\dots}(u) + \wp_{i,\dots}(\tilde{u}) + \wp_{i,\dots}(\hat{u}))$.
By expanding the right hand side of \eqref{TriLinRelGen} we find the trilinear relations 
in terms of $\mathfrak{s}_{i,\dots}$. Solving them, we obtain the addition formulas for $\zeta$-functions,
cf. \cite[Theorem 5.5]{BerLey2018},
\begin{align*}
\mathfrak{s}_{1} &= \gamma_1^{\text{o}}, \qquad\qquad 
\mathfrak{s}_{2} = 2 \gamma_2^{\text{o}}, \\
\mathfrak{s}_{5} &=  5 \gamma_5^{\text{o}} 
-  5 \gamma_3^{\text{o}}  \gamma_2^{\text{e}} 
-  5 \gamma_3^{\text{e}}  \gamma_2^{\text{o}} 
+ 5 \gamma_3^{\text{o}}  (\gamma_1^{\text{o}})^2 
+ 5  \big((\gamma_2^{\text{e}})^2 + (\gamma_2^{\text{o}})^2\big) \gamma_1^{\text{o}} 
- 5 \gamma_2^{\text{e}} \gamma_1^{\text{o}} \big( (\gamma_1^{\text{o}})^2 + \lambda_2 \big) \\
&\quad +  (\gamma_1^{\text{o}})^5 
+ \tfrac{2}{3} \lambda_2 \gamma_1^{\text{o}} \big(5 (\gamma_1^{\text{o}})^2 + \lambda_2\big)
- \tfrac{5}{8} \mathfrak{s}_{1,2,2} - \tfrac{1}{24} \mathfrak{s}_{1,1,1,1,1},
\end{align*}
and the addition law for basis $\wp$-functions.
Coefficients $\gamma$ are expressed in terms of basis $\wp$-functions at $u$ and $\tilde{u}$ 
in \eqref{GammaC27}.
\end{34Curve}

\subsection{Addition formulas for $\sigma$-function}
The first generalization of \eqref{AddLawSigma} to 
 the hyperelliptic case of genera two and three, $n=2$, were given by Baker in \cite[]{bakerHE1898}. 
 For example, in genus two we have
\begin{gather}
\frac{\sigma(u+\tilde{u}) \sigma(u-\tilde{u})}{\sigma(u)^2 \sigma(\tilde{u})^2}
= \wp_{1,1}(u) \wp_{1,3}(\tilde{u}) +  \wp_{3,3}(\tilde{u}) - \big( \wp_{1,1}(\tilde{u}) \wp_{1,3}(u) + \wp_{3,3}(u) \big).
\end{gather}
A generalization  to higher genera is suggested in \cite{BEL1997} in the concise form
\begin{equation}
\frac{\sigma(u+\tilde{u}) \sigma(u-\tilde{u})}{\sigma(u)^2 \sigma(\tilde{u})^2}
= \mathcal{S}(\tilde{u},u) - \mathcal{S}(u,\tilde{u}),
\end{equation}
and $\mathcal{S}$ are computed for a hyperelliptic curve of any genus.
Addition formulas for $\zeta$-functions can be obtained by applying operators $\mathfrak{D}$.

Further, $\mathcal{S}$ for some non-hyperelliptic and superelliptic curves have been computed:
\begin{itemize}
\item for a $(3,4)$-curve with extra terms in \cite[Theorem 9.1]{EEMOP2007};
\item for a  cyclic $(3,5)$-curve in \cite[Theorem 8.1]{BEGO2008};
\item for a  cyclic $(3,7)$-curve in  \cite[Theorem\,4, and Appendix\,B]{E2010}
\item for a  cyclic $(4,5)$-curve in \cite[Appendix D]{EE2009}.
\end{itemize}

A more general than \eqref{AddLawSigma} addition formula for the Weierstrass $\sigma$-function
\begin{multline}\label{FSnF}
(-1)^{(n-1)(n-2)/2} 1! 2! \cdots (n-1)!
\frac{\sigma \big( \sum_{k=1}^n u^{[k]} \big) \prod_{i<j} \sigma(u^{[i]}-u^{[j]})}
{\prod_{k=1}^n  \sigma(u^{[k]})^{n} }\\
 = \begin{vmatrix}
1 & \wp(u^{[1]}) & \wp'(u^{[1]}) & \dots & \wp^{(n-2)}(u^{[1]}) \\
1 & \wp(u^{[2]}) & \wp'(u^{[2]}) & \dots & \wp^{(n-2)}(u^{[2]}) \\
\vdots & \vdots & \vdots & \cdots &  \vdots\\
1 & \wp(u^{[n]}) & \wp'(u^{[n]}) & \dots & \wp^{(n-2)}(u^{[n]}) 
\end{vmatrix}
\end{multline}
can be found  in \cite{FS1877}.
The formula for all $u^{[k]}$ equal is known as the Kiepert formula, see \cite{kiepert1873}.
The function $\sigma (n u) / \sigma(u)^{n^2}$ is obtainable by a limiting process from \eqref{FSnF},
and defines division polynomials, see \cite{uchida2011}.

A generalization of \eqref{FSnF} 
written in terms of monomials $(x^i y^j)(u^{[k]})$ parameterized by $u^{[k]} \in \mathcal{A}(\mathfrak{C}_1)$
is given in \cite{On2002,On2004,On2005} for hyperelliptic curves of genera two, three,
and an arbitrary genus, respectively. An explicit expression 
in terms of $\wp$-functions on $\Jac(\mathcal{C}) \backslash \Sigma$
in the case of genus two and $n=3$ is obtained in  \cite{EEP2003}.

\section{Discussion}
Abelian function fields appear to be a perfect tool for studying Jacobian varieties.
Such a field provides (i) uniformization of a curve, which is well understood and computed now,
(ii) algebraic models of the Jacobian and Kummer varieties, (iii) connection between
differentiation rules and an algebraic structure of the ring of functions, (iv)
addition laws.

The presented construction of  abelian function fields was developed on $(n,s)$-curves,
where it arises naturally due to an exceptional role of the point at infinity, which is also a good choice of the basepoint.
$(n,s)$-Curves remains not fully understood, and so underestimated.
They are often considered as a special class of curves. On the contrary, they give a direct model of
the Weierstass canonical form of plane algebraic curves, which is clearly seen from \cite[\S\S 60--63]{bakerAF}.
The wrong impression is caused by the absence in the literature of examples of degenerate $(n,s)$-curves,
which fill the gaps in establishing connection between Weierstrass gap sequences, including
truncated, and the curves.

A solution of the Jacobi inversion problem, obtained in \cite{BLJIP22},
shows, that the number of relations which define the Jacobian variety
increases linearly with the genus  of a curve. This is a great improvement in comparison with the construction
based on theta identities, see \cite[\S 10]{mumfordIII}, where
the number of equations defining a Jacobian variety  increases exponentially.

Curves of gonalities greater than $3$ deserves more attention,  especially now, when it became clear
how to construct a solution of the Jacobi inversion problem on such curves.
Discovering the ring structure of abelian function fields induced by identities for $\wp$-functions, and 
clarifying uniformuzation of such curves remain open problems. 

The approach with polylinear operators is a virgin territory; it is
full of open problems and new horizons. Hierarchies of polylinear relations
definitely have finite amount of generating operators, and carry a distinct 
similarity within families of canonical curves.

It is worth to mention a progress in computation of $\wp$-functions.
As a solution of the KdV equation $\wp_{1,1}$-function is computed and graphically
represented in \cite{MatsKdV2024,BerKdV2024}. Direct computation of uniformization
is suggested in \cite{Ber2024} with $(2,9)$ and $(3,4)$-curves as examples.

\vspace{6pt}


\end{document}